\documentclass{amsart}

\usepackage{amsmath, amsthm}
\usepackage{amscd}
\usepackage{amssymb}
\usepackage{amsfonts}
\usepackage{mathrsfs}
\usepackage{latexsym}
\usepackage{graphicx}
\usepackage[all]{xy}
\usepackage{setspace}
\usepackage[usenames,dvipsnames]{color}
\usepackage{graphicx}
\usepackage{tikz}
\usetikzlibrary{matrix,arrows}
\usepackage{extarrows}

\newtheorem{theorem}{Theorem}[section]
\newtheorem*{thm}{Theorem}
\newtheorem{lemma}[theorem]{Lemma}
\newtheorem{prop}[theorem]{Proposition}
\newtheorem{cor}[theorem]{Corollary}

\theoremstyle{definition}
\newtheorem{example}{Example}[section]
\theoremstyle{definition}
\newtheorem{definition}{Definition}[section]

\newtheorem*{remark}{Remark}

\newcommand{\sdbar}{\overline{SD}}

\newcommand{\sdbarr}{\overline{\mathscr{SD}}}

\newcommand{\gkl}{(g, k, \ell)}

\newcommand{\s}{\sigma}
\newcommand{\e}{\varepsilon}
\newcommand{\M}{\hbox{Maps}}
\newcommand{\MM}{{\mathcal Maps}}
\newcommand{\ks}{\sqcup_k S^1}
\newcommand{\els}{\bigsqcup_\ell S^1}
\newcommand{\ksd}{\bigsqcup_k S^1 \times \Delta^n}

\newcommand{\st}{g}
\newcommand{\be}{\begin{enumerate}}
\newcommand{\ee}{\end{enumerate}}
\newcommand{\olo}{\otimes \ldots \otimes}
\renewcommand{\to}{\longrightarrow}
\renewcommand{\mapsto}{\longmapsto}
\renewcommand{\setminus}{-}
\newcommand{\circc}{\stackrel{\; \circ}{c}}
\newcommand{\abs}[1]{|\,#1\,|}

\title{ Compactifying string topology }
\author[K. Poirier]{Kate Poirier}
\address[Kate Poirier]{Department of Mathematics, University of California, Berkeley, CA, 94720}
\email{poirier@math.berkeley.edu}
\urladdr{www.math.berkeley.edu/$\sim$poirier}
\author[N. Rounds]{Nathaniel Rounds}
\address[Nathaniel Rounds]{Department of Mathematics, University of Indiana, Bloomington, IN, 47405}
\email{nrounds@indiana.edu}
\urladdr{mypage.iu.edu/$\sim$nrounds}

\begin{document}

\begin{abstract}
We study the string topology of a closed oriented Riemannian manifold $M$.
We describe a compact moduli space of diagrams, called $\sdbar$, and show how the cellular chain complex of this space gives algebraic operations on the singular chains of the free loop space $LM$ of $M$.
These operations are well-defined on the homology of a quotient of this moduli space, $\sdbar/\sim$, which has the homotopy type of a compactification of the moduli space of Riemann surfaces.
In particular, our action of $H_0(\sdbar/\sim)$ on $H_*(LM)$ recovers the Cohen-Godin positive boundary TQFT on $H_*(LM)$.
\end{abstract}

\maketitle
\setcounter{tocdepth}{1}
\section*{Introduction}
String topology studies algebraic operations on the loop space of a manifold.
Let $M$ be a closed oriented smooth  manifold of dimension $d$, and let $H_*(LM)$ denote the singular homology of the free loop space $LM=Maps(S^1,M)$.
Chas and Sullivan constructed a {\em loop product}
\[ \bullet: H_i(LM) \otimes H_j(LM) \to H_{i+j-d}(LM)\]
and BV operator
\[ \Delta: H_i(LM) \to H_{i+1}(LM)\]
giving $H_*(LM)$ the structure of a Batalin-Vilkovisky algebra~\cite{CS}.
Let $H_*^{S^1}(LM)$ denote the $S^1$-equivariant homology of $LM$.
Chas and Sullivan also constructed a {\em string bracket} 
\[ \{\, , \,\}: H_i^{S^1}(LM) \otimes H_j^{S^1}(LM) \to H_{i+j-d+2}^{S^1}(LM)\]
giving $H_*^{S^1}(LM)$ the structure of a graded Lie algebra. 
They later extended structure this to an involutive Lie bialgbra structure on $H^{S^1}(LM,M)$ where $M$ denotes the subspace of constant loops in $LM$ \cite{CS2}.

These algebraic structures have since been explained and generalized in many ways.
The modern view is that string topology operations should be parameterized by moduli spaces of Riemann surfaces.
Spaces of fatgraphs have long been used to give combinatorial descriptions of the open moduli space of Riemann surfaces \cite{Strebel, Penner, Harer, Igusa, costello2}.

Cohen and Jones \cite{cohenjones} gave a homotopy-theoretic reformulation of the Chas-Sullivan product which Cohen and Godin generalized using a class of fatgraphs called {\em Sullivan chord diagrams} \cite{CG}.
The Cohen-Godin operations induce an action of $h_0(Sull\gkl)$ on $h_*(LM)$, where $Sull\gkl$ is the space of Sullivan chord diagrams and $h_*$ is any homology theory for which $M$ has an orientation. 
This action gives $H_*(LM)$ the structure of a Frobenius algebra with no counit, which Cohen and Godin called a positive boundary Topological Quantum Field Theory.
Chataur's extended this action to one of $H_*(Sull\gkl)$ on $H_*(LM)$ \cite{ChataurBordism}.

The space $Sull\gkl$ of Sullivan chord diagrams is a subspace of the moduli space $\mathcal{M}\gkl$ of Riemann surfaces of genus $g$ with $k$ incoming and $\ell$ outgoing boundary components. Cohen and Godin conjectured that  $Sull\gkl$ and $\mathcal{M}\gkl$ have the same homotopy type. Godin discovered that this conjecture is false \cite{godin}, and generalized string topology operations further to give an action of $H_*(\mathcal{M}\gkl)$ on $H_*(LM)$. She calls this structure a Homological Conformal Field Theory.

In the case where $M$ is simply connected, string topology operations can be studied from the perspective of Hochschild homology
 \cite{CV, CTZ, FTVP, Kaufmann,TZ}. Westerland and Wahl recently described an action on the Hochschild homology recovering the Chas-Sullivan BV structure as part of a Homological Conformal Field Theory \cite{ww}. It is not known if this HCFT agrees with the one Constructed by Godin.

It is expected that the string topology operations described above are the shadow of a deeper structure, for
all of these results can potentially be generalized in two directions. 
First, the Cohen-Godin-Chataur action of $H_*(Sull\gkl)$ in $H_*(LM)$ should be induced by an chain-level action of $C_*(Sull\gkl)$ on $C_*(LM)$.
Different flavors of this idea can be described using the language of Open-Closed Topological Conformal Field Theories in the sense of Getzler~\cite{getzler} and Costello~\cite{costello}, and the language of Topological Quantum Field Theories in the sense of Moore-Segal~\cite{segal} and Lurie~\cite{lurie}.
Blumberg, Cohen, and Teleman have recently made progress in describing string topology in this way~\cite{bct}.

Second, the space $Sull\gkl$ is a subspace of an open moduli space of Riemann surfaces.
Sullivan has conjectured that a compactification of the open moduli space of Riemann surfaces should act on $H_*(LM)$ and $C_*(LM)$~\cite{STBAPS}.
Our goal is to give an action of the cellular chains of a compactified moduli space of Riemann surfaces on the singular chains of the free loop space.
This paper constitutes a first step towards this goal.

Instead of the space $Sull\gkl$ of Sulllivan chord diagrams, we study a related space $\sdbar\gkl$ of {\em string diagrams}. 
The main result is the following.
\begin{thm}
Let $M$ be a closed, oriented, Riemannian manifold of dimension $d$, and let  $\sdbar(g,k,\ell)$ be the cellular moduli space of string diagrams of type $(g,k,\ell)$.
There exists a chain map
\[ \mathcal{ST}: \mathcal{C}_i(\sdbar(g,k,\ell)) \otimes C_j(LM) \to C_{i+j + (2-2g-k-\ell)d}(LM).\]
This chain map induces a map on homology:
\[ \mathcal{ST}: \mathcal{H}_i(\sdbar\gkl/\sim) \otimes H_j(LM) \to H_{i+j+(2-2g-k-\ell)d}(LM).\]
When $i=0$, the resulting maps
\[ \mathcal{ST}: \mathcal{H}_0(\sdbar\gkl/\sim) \otimes H_j(LM) \to H_{j+(2-2g-k-\ell)d}(LM).\]
recover Cohen and Godin's positive boundary TQFT structure on $H_*(LM)$.
\end{thm}

We now summarize the contents of the paper.
A {\em string diagram of type $(g,k,\ell)$} is a certain type of fatgraph which determines a Riemann surface of genus $g$ with $k+\ell$ boundary components.
In Section 1, we define for each $g \ge 0$, $k>1$,and $\ell >1$ a compact moduli space $\sdbar(g,k,\ell)$ of string diagrams, and describe a CW complex structure on this moduli space. $\sdbar(g,k,\ell)$.
We also define an open subspace $SD\gkl$ of $\sdbar\gkl$ which is a union of open cells.

Let $M$ be a closed, oriented Riemannian manifold of dimension $d$, let $C_*(LM)$ denote the singular chain complex of $LM$, and let $\mathcal{C}_*(\sdbar(g,k,\ell))$ denote the cellular chain complex of $\sdbar(g,k,\ell)$.
In Section 2, we define a map we call the string topology construction:
\[ \mathcal{ST}: \mathcal{C}_*(\sdbar(g,k,\ell)) \otimes C_*(LM) \to C_{* + (2-2g-k-\ell)d}(LM).\]

In Section 3, we prove that $\mathcal{ST}$ is a chain map.

In Section 4, we put an equivalence relation $\sim$, called {\em slide equivalence}, on the cells of $\sdbar\gkl$, and prove that $SD\gkl/\sim$ is homotopy equivalent to $Sull\gkl$.
Thus, $\sdbar\gkl/\sim$ is a compactification of a space homotopy equivalent to $Sull\gkl$.
It is in this sense that we are compactifying string topology.
The cell complex $\sdbar\gkl/\sim$ was shown in the first author's thesis to be homotopy equivalent to B\"odigheimer's harmonic compactification of the open moduli space of Riemann surfaces of type $\gkl$.~\cite{bodigheimer,katethesis}.
The string topology construction is not well-defined on slide equivalence classes of string diagrams.
However, we show that if two cells $c$ and $c'$ of $\sdbar\gkl$ are slide equivalent, then the maps $\mathcal{ST}(c,-)$ and $\mathcal{ST}(c',-)$ differ by a chain homotopy.
Thus $\mathcal{ST}$ gives a well-defined map 
\[ \mathcal{ST}: \mathcal{H}_*(\sdbar\gkl/\sim) \otimes H_*(LM) \to H_{*+(2-2g-k-\ell)d}(LM).\]
We show that this map recovers Cohen-Godin's action of $H_0(Sull\gkl)$.
We do not know if the operations coming from the higher homology of $\sdbar\gkl/\sim$ agree with those of Chataur or with those of Godin.

In Section 5, we prove a gluing result to show that our action of $H_0(\sdbar\gkl/\sim~)$ gives a Frobenius algebra without counit in the sense of Cohen-Godin. Furthermore, the homotopy equivalence of Corollary 4.9 induces an isomorphism between this Frobenius algebra structure and that of Cohen-Godin.
 
One might wish to say that $C_*(\sdbar\gkl/\sim)$ is a PROP or a properad, and that we have an action of this properad on $C_*(LM)$.
However, this more ambitious claim is false for two reasons.
This first, alluded to above, is that our string topology construction differs by a chain homotopy on slide equivalent cells, so after we quotient by slide equivalence our operations are well-defined only on homology.
The second problem, discussed in Section 5,
is that gluing of string diagrams induces composition maps on $C_*(\sdbar\gkl/\sim)$, but these maps are associative only up to homotopy.
In a future paper, we plan to construct a larger space, $\overline{\mathscr{LD}}\gkl$ which is homotopy equivalent both to $\sdbar\gkl/\sim$ and to B\"odigheimer's harmonic compactification of the moduli space of Riemann surfaces. The cellular chains of $\overline{\mathscr{LD}}\gkl$
will form a properad under gluing of surfaces, 
and we plan to show that this properad acts on $C_*(LM)$.


The original homology-level operations of Chas-Sullivan relied on transversality assumptions. The idea of using short geodesic arcs to give a chain level string topology construction, as carried out in Section 2 of this paper, was first suggested by Dennis Sullivan \cite{STBAPS}. This geodesic construction allows us to define chain-level string topology operations without making transversality assumptions.

A construction similar to the string topology construction of Definition~\ref{def stc} appeared in the first author's thesis \cite{katethesis}.

\subsection*{Acknowledgements.}
The authors would like to thank Dennis Sullivan and Janko Latschev  for many helpful conversations.

The first author is partially supported by NSF RTG grant DMS-0838703.

\section{The space of string diagrams}\label{sdbarUSD}

In this section we define a class of graphs with extra structure, called string diagrams, and show that the moduli space $\sdbar$ of string diagrams is a CW complex. 
We then describe a second CW complex, called $U\sdbar$, and a projection map from $U\sdbar$ to $\sdbar$ that the fiber in $U\sdbar$ over every point in the moduli space $\sdbar$ is the string diagram corresponding to that point. 
Though the projection map is not a bundle map, we think of $U\sdbar$ as a ``universal bundle'' over the moduli space $\sdbar$.

In the sequel, $S^1$ denotes the standard oriented metric graph with one vertex and one edge of length 1:
\[S^1 = [0,1] / 0 \sim 1.\]

\subsection{Fatgraphs and string diagrams}

\begin{definition}
A {\em fatgraph} is the follwowing data:
\be
\item A finite connected graph $\Gamma$.
\item For each vertex $v$ of $\Gamma$, a cyclic order of the set of edges adjacent to $v$.
\ee
By a cyclic order of a set, we mean a permutation of that set which is a single cycle.
\end{definition}

 \begin{figure}[h] \label{fatgraph1}
\centering
\includegraphics{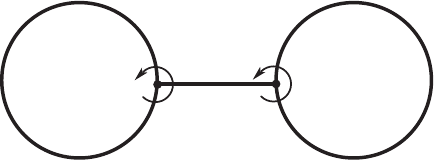}
\caption{A fatgraph with two vertices and three edges. The cyclic orders are indicated.}
\end{figure}

\begin{definition}
Let $\Gamma$ be a fatgraph, and let $E_\Gamma$ denote the set of edges of $\Gamma$.
Let $\overline{E}_\Gamma$ denote the set
\[ \{ (e,o) \,\, | \,\, e \in E_\Gamma, o \text{ is an orientation of } e\}. \]
Then the cyclic order of the set of edges adjacent to each vertex of $\Gamma$ induces a permuation $\sigma_\Gamma$ of the set $\overline{E}_\Gamma$ defined as follows.
Let $(e,o)$ be an oriented edge with final vertex $v$.
Let $e'$ be the next edge after $e$ in the cyclic order of the edges adjacent to $v$.
Let $o'$ be the orientation of $e'$ for which $v$ is the intitial vertex of $e'$.
Then we set $\sigma_\Gamma(e,o) = (e',o')$.
A {\em boundary cycle} $\beta$ of $\Gamma$ is a cycle of oriented edges in the permutation $\sigma_\Gamma$.
The {\em realization} of a boundary cycle, denoted $\abs{\beta}$, is the oriented graph homeomorphic to a cicle whose cyclically ordered set of edges is precisely the set $\beta$.
\end{definition}

A fatgraph determines an orientable topological surface with boundary $\Sigma_\Gamma$ which contains the underlying graph as a deformation retract \cite{godin}.
This topological surface, sometimes called a ribbon surface, may be constructed as follows.
Starting with the underlying graph $\Gamma$, thicken the vertices $v$ into disks $D_v$ and the edges $e$ into strips $e \times I$. 
If $e$ is adjacent to $v$ in $\Gamma$, then the corresponding boundary component of $e \times I$ is identified with an arc on $\partial D_v$ in $\Sigma_\Gamma$. 
Boundaries of strips are identified along $\partial D_v$ according to the cyclic order of the corresponding edges adjacent to $v$. 
The boundary cycles of $\Gamma$ are in one-to-one correspondence with the boundary components of $\Sigma_\Gamma$. 
\begin{definition}
A fatgraph $\Gamma$ is of {\em type} $(g,n)$ if $\Sigma_\Gamma$ is of genus $g$ with $n$ boundary components.
\end{definition}

\begin{figure}[h] \label{fatgraph2}
\centering
\includegraphics{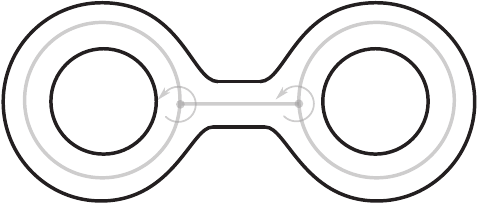}
\caption{The ribbon surface associated to a fatgraph.}
\end{figure}

\begin{definition}
A {\em metric fatgraph} is a fatgraph whose underlying graph $\Gamma$ is a metric space.
A {\em marked metric fatgraph} is a metric fatgraph together with a marked basepoint $0_\beta \in \abs{\beta}$ on the realization of each boundary cycle.
\end{definition}

Let $\Gamma$ be a marked metric fatgraph and let $\beta$ be a boundary cycle of $\Gamma$.
Let $\ell_\beta$ denote the sum of the lengths of edges which appear in the cycle $\beta$.
Let 
\[ S^1 \xrightarrow{\text{rescale}} [0,\ell_\beta]/(0 \sim \ell_\beta)\]
denote unique linear map which rescales the interval $[0,1]$ onto the interval $[0,\ell_\beta]$.
Let 
\[\phi_\beta: [0,\ell_\beta]/(0\sim \ell_\beta) \to \abs{\beta}\]
be the unique orientation preserving isometry of metric circles which sends 0 to the marked point $0_\beta$ of $\abs{\beta}$.
Let 
\[ \psi_\beta: \abs{\beta} \to \Gamma\]
be the unique map which sends sends the oriented edge $(e,o)$ of $\abs{\beta}$ bijectively onto the edge $e$ of $\Gamma$ in a manner which respects the orientation.
Let $\partial_\beta$ denote the composition
\[ \partial_\beta: S^1 \xrightarrow{\text{rescale}}  [0,\ell_\beta]/(0\sim \ell_\beta) \xrightarrow{\phi_\beta}  \abs{\beta} \xrightarrow{\psi_\beta} \Gamma.\]

Let $\Gamma$ be a marked metric fatgraph of type $(g,n)$ and let $\beta_1, \beta_2, \dots, \beta_n$ be its boundary cycles. 
Let 
\[ \partial_\Gamma: \sqcup_n S^1 \to \Gamma\]
denote the map which restricts to $\partial_{\beta_i}$ on the $i$-th copy of $S^1$.

\begin{definition}
An {\em unordered string diagram of type $\gkl$} is a marked metric fatgraph $\Gamma$ of type $(g, k + \ell)$ that is constructed from $k$ disjoint circles, called input circles, each of length $1$, and $2g-2+k + \ell$ intervals, called chords, each of length $1$. 
The endpoints of a chord $e$ are identified with points on input circles via an attaching map $\varphi_e: \partial(I) = \{0,1\} \to \sqcup_k S^1$. 
The cyclic order of edges at each vertex of $\Gamma$ is such that $k$ of the boundary cycles correspond to the input circles. 
The remaining $\ell$ boundary cycles are called output circles. 

A {\em string diagram of type $\gkl$} is an unordered string diagram of type $\gkl$ together with an ordering of the set of input circles and an ordering of the set of output circles.
\end{definition}

In Figures 3 and 4, vertices are denoted by $\bullet$ and marked points on boundary cycles are denoted by $\times$.

 \begin{figure}[h]
\centering
\includegraphics{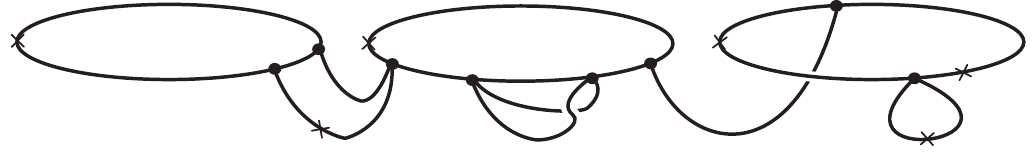}\label{stringdiagram}
\caption{A string diagram of type $(1,3,3)$.}
\end{figure}

 \begin{figure}[h]
\centering
\includegraphics{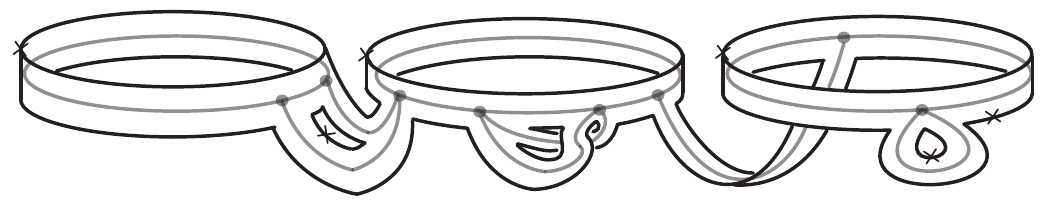}\label{stringdiagramfat}
\caption{The ribbon surface associated to the string diagram above.}
\end{figure}

\begin{remark}
A graph which is a disjoint union of $k$ circles has Euler characteristic $0$. Attaching the endpoints of a chord to a graph decreases the Euler characteristic by $1$ so a string diagram $\Gamma$ of type $\gkl$ has Euler characteristic  $-(2g-2+k + \ell)$. As Euler characteristic is a homotopy invariant, the ribbon surface $\Sigma_\Gamma$ associated to $\Gamma$ also has Euler characteristic  $-(2g-2+k + \ell)$. In particular, the Euler characteristic of $\Sigma_\Gamma$  is the Euler characteristic of a surface of genus $g$ with $k+\ell$ boundary components.
\end{remark}

\begin{definition}
A morphism of string diagrams is a map of the underlying metric graph that preserves cyclic orders of edges and markings of boundary cycles.
\end{definition}

In what follows, by {\em string diagram} we mean {\em isomorphism class of string diagrams}.

\begin{remark}\label{emailedtokate}
Each input circle of $\Gamma$ is an oriented metric circle of length 1. 
When $\beta$ is an input boundary cycle, $\ell_\beta = 1$. The identification of $|\beta|$ with $S^1$ is therefore an isometry and it is uniquely determined by the orientation and marked point of $|\beta|$. Additionally, the image of the map $\partial_\beta: S^1 \to \Gamma$ is an input circle of $\Gamma$ and $\partial_\beta$ is an
isometry.
In what follows, we will suppress the distinctions between
the realization $|\beta|$ of
 an input boundary cycle $\beta$ and the corresponding input circle which occurs as the image of $\partial_\beta$.
\end{remark}

Let $\alpha: S^1 \to S^1$
be the unique orientation {\em reversing} isometry taking the $0$-cell of $S^1$ to itself.
While the realization $|\beta|$ of the input circle $\beta$ is parametrized by $S^1$ using $\phi_\beta$,
we parametrize the realization $|\beta|$ of the output circle $\beta$ by $S^1$ using the composition
$$\phi_\beta \circ {\rm rescale} \circ \alpha: S^1 \to |\beta|.$$

\begin{remark}
The combinatorial data associated to a string diagram $\Gamma$ determines an ordering of the set of chords of $\Gamma$, which we describe in three stages.
\begin{enumerate}
\item The set of half-chords adjacent to each vertex $v$ of $\Gamma$ are ordered as follows.
Part of the data of a string diagram is a cyclic order of the half-edges adjacent to $v$.
Since $v$ lies on some input circle and each input circle is a boundary cycle, the two half-edges adjacent to $v$ which lie on the input circle must be adajecent in the cyclic order.
We order the half-chords adjacent to $v$ by starting with the first half-chord which follows the two half-edges which lie on the input circle in the cyclic order, and then proceeding with the cyclic order.
\item Each input circle is an oriented circle with a marked point.
We order the set of vertices on each circle by starting with the vertex closest to the marked point in the direction of the orientation and then proceeding in the direction of the orientation.
\item Part of the data of a string diagram is an ordering of the input circles.
Combining (1), (2), and (3) gives an ordering of the set of half-chords of $\Gamma$.
We order the set of chords by remembering only the first time that a half-chord of a given chord appears in this order.
\end{enumerate}
\end{remark}

\subsection{The space of string diagrams}

The set of isomorphism classes of string diagrams is a subset of the set of isomorphism classes of marked metric fatgraphs. 
The set of marked metric fatgraphs is given a topology \cite{Harer, Penner, Igusa}. 
The set of string diagrams inherits the subspace topology.

In what follows, we fix $\gkl$ where $g \geq 0$ and $k, \ell > 0$.

\begin{definition}
Let $\sdbar$ be the space of string diagrams of type $\gkl$. Let $SD$ be the subspace of $\sdbar$ consisting of string diagrams each of whose subgraph of chords is a disjoint union of trees.
Let $x_\Gamma$ denote the point in $\sdbar$ corresponding to the string diagram $\Gamma$.
\end{definition}

\begin{remark}
We are emphasizing the distinction between a string diagram $\Gamma$ as a topological space and its corresponding point $x_\Gamma \in \sdbar$.
Below we construct a space $U\sdbar$ and a surjective map $\pi: U\sdbar \to \sdbar$ so that $\pi^{-1}(x_\Gamma) = \Gamma$.
\end{remark}

\begin{example}
$\sdbar(0,2,1)$ is homeomorphic to $T^3$ the $3$-dimensional torus. Every string diagram of type $(0,2,1)$ consists of two input circles and one chord. The three circle parameters correspond to placements of marked points on the two input circles and one output circle.
\end{example}

\begin{prop}\label{sdbarCW}
The space $\sdbar$ is a CW complex of dimension $4g-4 + 2k + 3 \ell$.
\end{prop}

\begin{proof}

A string diagram $\Gamma$ has an underlying fatgraph $G$ obtained by forgetting the metric structure. The collection of edges of $G$ are partitioned into input circle edges and chords as they are in $\Gamma$. 
We say that two string diagram have the same {\em combinatorial type} if their underlying marked fatgraphs are isomorphic.
An isomorphism $G \to G'$ of marked fatgraphs preserves the cyclic order at each vertex, and thus induces isomorphisms $\abs{\beta} \to \abs{\beta'}$ for each boundary cycle $\beta$ of $G$.
In particular, the location of the marked point on each boundary cycle --- either coinciding with a vertex of $\abs{\beta}$ or lying in the interior of a directed edge of $\abs{\beta}$ --- is preserved by this isomorphism.
Henceforth, $G$ denotes an isomorphism class of marked metric fatgraphs giving a combinatorial typle of string diagrams of type $\gkl$.

We show first that for a fixed $G$, the subspace $$\circc_G = \{x_\Gamma \in \sdbar \; |\textrm{ the fatgraph underlying }\Gamma \textrm{ is } G\}$$ of $\sdbar$ forms an open cell.

For a fixed $G$, a string diagram $\Gamma$ of combinatorial type $G$ is completely determined by the following parameters:
\begin{enumerate}
\item
The positions of the vertices on input circles of $G$.
\item
The positions of the marked points on the boundary cyles of $G$.
\end{enumerate}
We will see that $\circc_G$ is a product of $k$ open simplices and $N_\ell \leq \ell$ open intervals.

Consider $\Gamma$ with underlying fatgraph $G$. 
Let $n_i$ denote the number of vertices on the $i$-th inputs circle of $G$ which do not coincide with the marked point on that input circle.
If $n_i > 0$ for some $i$, let 
$v_1, v_2, \dots v_{n_i}$ denote these vertices on the $i$-th input circle. 
Let $t^i_0$ be the distance from the marked point to $v_1$, $t^i_j$ be the distance from $v_j$ to $v_{j+1}$ for $j = 1, 2, \dots, n_{i-1}$ and let $t^i_{n_i}$ be the distance from $v_{n_i}$ to the input marked point. Then $\sum_{j = 0}^{n_i} t^i_j = 1$. In particular, the positions of vertices on the $i$-th input circle of $\Gamma$ are determined by the point $t^i=(t^i_1, t^i_2, \dots, t^i_{n_i})$ in the interior of the standard $n_i$-simplex $\Delta^{n_i}$.
If $n_i=0$ for some $i$, then there are no vertices and the previous sentence is still true, provided we define the interior of 0-simplex to be the 0-simplex itself.
Therefore, the positions of vertices on all $k$ input circles of $\Gamma$ are determined by a point $t = (t^1, t^2, \dots, t^k)$ in the interior of $\Delta^{n_1} \times \Delta^{n_2} \times \cdots \times \Delta^{n_k}$.

Suppose that the marked point on the $i$-th output boundary cycle of $\Gamma$ lies in the interior of a directed edge $\vec{e_i}$. Let $[0,1]$ parametrize $\vec{e_i}$ such that $0$ maps to the source of $\vec{e_i}$ and $1$ maps to its target. Then the point marking $\vec{e_i}$ is determined by a point $p_i \in (0,1)$. 
Let $N_\ell$ be the number of output marked points on $\Gamma$ that lie in the interiors of directed edges. 
Then the positions of such marked points are determined by a point $p \in (0,1)^{N_\ell}$.

As the parameters $(t,p)$ vary in $int(\Delta^{n_1} \times \Delta^{n_2} \times \cdots \times \Delta^{n_k}) \times (0,1)^{N_\ell}$, all string diagrams with underlying fatgraph $G$ are obtained. Therefore,
$$\circc_G = int(\Delta^{n_1} \times \Delta^{n_2} \times \cdots \times \Delta^{n_k}) \times (0,1)^{N_\ell}$$
which is homeomorphic to the interior of 
$$c_G = \Delta^{n_1} \times \Delta^{n_2} \times \cdots \times \Delta^{n_k} \times [0,1]^{N_\ell}$$
This space is homeomorphic to a closed ball, i.e., a cell.
We have $$\sdbar = \bigsqcup_G \circc_G,$$
where $G$ varies over combinatorial types of string diagrams of type $\gkl$.

Next we describe how the closed cells $c_G$ are assembled to give a CW-complex.

The $0$-cells of $\sdbar$ correspond to combinatorial types $G$ where all vertices and all output marked points coincide with input marked points.

Let $\sdbar^{\;m}$ be the $m$-skeleton of $\sdbar$. For $m$-dimensional cells, the attaching maps $\phi_G: \partial(c_G) \rightarrow \sdbar^{\; m-1}$ are determined by identifications of faces of the cell $c_G$ with cells $c_{G'}$ of lower dimension. If $(t,p)$ is a point on the boundary of $c_G$ then either some $t^i_j = 0$, or some $p_i = 0$ or $1$.

Let $(t,p)$ be a point in $c_G$ such that $t^i_j = 0$ for one $j \in \{1, 2, \dots, n_i - 1\}$. Then $(t,p)$ determines some string diagram $\Gamma'$ the vertices $v_j$ and $v_{j+1}$ coincide. In $\Gamma'$ this vertex is labeled $v_j$ and the vertices $v_{j+2}, \dots, v_{n_i}$ are renumbered accordingly.
The combinatorial type $G'$ of $\Gamma'$ is obtained from $G$ by contracting an input circle edge between vertices. Similarly, if $t_0 = 0$ (respectively $t_{n_i} = 0$) then $G'$ is obtained from $G$ by bringing the first (respectively last) vertex and the marked point on the $i$-th input circle together. Finally, if $p_i = 0$ (respectively $p_i=1$) for some $i$, then $G'$ is obtained from $G$ by bringing the output marked point in the interior of $\vec{e_i}$ to its source (respectively its target). In each of these cases, $c_{G'}$ is identified with the appropriate face of the cell $c_G$. These identifications determine the attaching map $$\phi_G: \partial(c_G) \to \sdbar^{\; m-1}$$
where $\phi_G(t,p) = x_{\Gamma'}$.

Top-dimensional cells $c_G$ correspond to combinatorial types $G$ such that no chord endpoints coincide with one another or with marked points on the boundary cycles. 
Each chord gives 2 parameters: the positions of its endpoints on input circles relative to the input marked point. 
Each output marked point contributes 1 parameter: its position on its output boundary cycle. 
As there are $2g -2 +k + \ell$ chords and $\ell$ output boundary cycles, 
\begin{align*}
dim(c_G) &= 2(2g-2 +k + \ell) + \ell\\ &= 4g-4 + 2k + 3 \ell.
\end{align*}
\end{proof}

\begin{remark}
Every cell in $\sdbar$ is the face of some top-dimensional cell.
\end{remark}

\begin{cor}
The subspace $SD$ of $\sdbar$ is a union of open cells and is dense in $\sdbar$.
\end{cor}

\begin{proof}
Given a string diagram $\Gamma$, the {\em chord subgraph} is the subgraph of the underlying fatgraph of $\Gamma$ which consists only of the chords.
Recall that $SD$ is the subsapce of $\sdbar$ consisting of those string diagrams whose chord subgraphs are disjoint unions of trees.
Let $\Gamma$ be a string diagram in $SD$.
Then $x_\Gamma \in \circc_G$, where $G$ denotes the combinatorial type of $\Gamma$.
Furthermore, for all $x_{\Gamma'} \in \circc_G$, the chord subgraph of  $\Gamma'$ is a disjoint union of trees. 
Let 
$ \mathcal{G}$ denote the set of all combinatorial types of string diagrams in $\sdbar$ such that the chord subgraph of $G$ is a disjoint union of trees.
Then
$$SD = \bigsqcup_{G \in \mathcal{G}} \circc_G.$$

Let $c_G$ be a top-dimensional cell of $\sdbar$.
If $x_\Gamma \in \circc_G$ then all chords of $\Gamma$ have distinct endpoints and its chord subgraph is a disjoint union of trees. Therefore, all top-dimensional cells are in $SD$. Since every cell of  is the face of some top-dimensional, $SD$ is dense in $\sdbar$.
\end{proof}

\begin{example}
Recall that $\sdbar(0,2,1)$ is homeomorphic to $T^3$. The first two $S^1$ factors correspond to the placements of chord endpoints on the two input circles relative to the marked points. These factors are decomposed according to the standard CW decomposition of $S^1$ with one $0$-cell and one $1$-cell. The third $S^1$ parameter corresponds to the placement of the point marking the output. The marked point may lie in one of four possible regions: one one or the other input circle or on one or the other directed edge coming from the chord. Therefore, the third $S^1$ parameter is decomposed into four $0$-cells and four $1$-cells.
\end{example}

This CW structure on $\sdbar$ is not a regular cell complex structure. Let $c_G$ be an $m$-cell and let $\Phi_G: c_G \to \sdbar$ be the characteristic map induced by the attaching map $\phi_G: \partial(c_G) \to \sdbar^{\; m-1}$. Then the closure of $\Phi(\circc_G)$ in $\sdbar$ is not homeomorphic to a closed ball. Let $(t,p) \in \partial(c_G)$ be such that $t^i_j = 0$ for all $j = 0, 1, \dots, n_i - 1$ and $t^i_{n_i} = 1$ for some $i$. Then $(t, p)$ determines a string diagram $\Gamma'$ where all chord endpoints on the $i$-th input circle coincide with the input marked point. Consider $(t', p) \in \partial(c_G)$ such that $t'^x_j = t^x_j$ for all $x \neq i$, $t^i_0 = 1$ and $t^i_j = 0$ for all $j = 1, \dots, n_i $. Then $(t', p)$ also determines a string diagram $\Gamma''$ where all chord endpoints on the $i$-th input circle coincide with the input marked point. In fact, cyclic orders of chords at this vertex agree and $\Gamma' = \Gamma''$ so $\phi_G(t,p) = \phi_G(t',p)$. Similarly, consider $G$ where the $i$-th output boundary cycle consists of a single directed edge $\vec{e}_i$. This is possible only if the two chord endpoints of the chord $e_i$ coincide on some input circle. Let  $\Gamma' \in \partial(c_G)$ be determined by $(t,p)$ where $p_i=0$ and let $\Gamma'' \in \partial(c_G)$ be determined by $(t,p')$ where $p'_x = p_x$ for $x \neq i$ and $p'_i = 1$.  Again, $\Gamma'=\Gamma''$ and $\phi_G(t,p) = \phi_G(t,p')$. 
We conclude that $\Phi(\circc_G)$ is homeomorphic to a product of $k$ quotients of simplices where the first and last vertices have been identified and $\ell$ quotients of intervals where the endpoints of $i$-th interval are identified if the $i$-th output boundary cycle is made up of one directed edge.

The definition of the string topology construction in section \ref{STC} uses a regular cell structure of $\sdbar$ which is a decomposition of the one above.

By replacing each simplex and interval factor in a cell $$c_G = \Delta^{n_1} \times \Delta^{n_2} \times \cdots \times \Delta^{n_k} \times [0,1]^{N_\ell}$$ of the CW structure of $\sdbar$ by its barycentric subdivision, we obtain a new decomposition of $c_G$ and hence one of $\sdbar$. 

\begin{lemma}
The CW structure on $\sdbar$ obtained by sudividing each cell $c_G$ as above is a regular cell complex structure.
\end{lemma}

\begin{proof}
Let $c_G = \Delta^{n_1} \times \Delta^{n_2} \times \cdots \times \Delta^{n_k} \times [0,1]^{N_\ell}$. We saw in the proof of proposition \ref{sdbarCW} that the characteristic map $\Phi_G: c_G \to \sdbar$ identifies faces of $c_G$ corresponding with the first and last 0-cell of each simplex factor and it identifies the endpoints of any interval factor corresponding to an output marked point lying in the interior of a directed edge where the output consists of a single chord. When each factor is replaced by its barycentric subdivision, the new cells are again products of simplices and intervals but in this decomposition, no two faces of a single cell are identified in $\sdbar$.
\end{proof}

In either decomposition, a cell $c$ is a product of simplices and intervals. By identifying the unit interval $[0,1]$ with the standard $1$-simplex $\Delta^1$, we see that $c$ is a product of simplices.

\subsection{The space $U\sdbar$}

In this section we construct a space $U\sdbar$ and a map $$\pi: U\sdbar \to \sdbar$$ such that for $x_\Gamma \in \sdbar$, $\pi^{-1}(x_\Gamma)$ is the string diagram $\Gamma$. 
We construct $U\sdbar$ and $\pi$ cell by cell.
We use the first cell decomposition of $\sdbar$ described in the previous section, in which cells are indexed by combinatorial types of string diagrams.

\begin{definition}
Consider the cell $c_G$ of $\sdbar$ labeled by the fatgraph $G$, and let $\{e\}$ denote the set of chords of $G$.
Let $(t,p) \in c_G$ determine a string diagram $\Gamma$ with a corresponding collection of chords $\{e_\Gamma\}$ and attaching maps $$\{\varphi_{e_\Gamma}: \{0, 1\} \to \sqcup_k S^1\}.$$ The cell complex
$U\sdbar(c_G)$ is given by
$$ \left( \left(c_G \times \bigsqcup_{\{e\}} I_e \right) \sqcup \left( c_G \times \bigsqcup_k S^1\right) \right)  / \sim$$ 
where for all $(t,p, 0)$ and $(t,p, 1) \in c_G \times  \partial (I_e)$, 
$$(t,p, 0) \sim (t,p, \varphi_{e_\Gamma}(0))  \in c_G \times \sqcup_k S^1$$
and
$$(t,p, 1) \sim (t,p, \varphi_{e_\Gamma}(1)) \in c_G \times \sqcup_k S^1.$$
\end{definition}

\begin{definition}
Let $\pi_{c_G}: U\sdbar(c_G) \to c_G$ be the map induced  by the projection:
$$\tilde{\pi}_{c_G}: \left(c _G \times\bigsqcup_{\{e\}} I_e \right) \sqcup \left( c_G \times \bigsqcup_k S^1\right)  \to c_G.$$
\end{definition}

Let $(t,p)$ be a point in $c_G$.
Then $(t,p)$ determines a string diagram $\Gamma$,  and  $\pi_{c_G}^{-1}(t,p)$ is the metric graph underlying $\Gamma$. 
We wish to identify $\pi_{c_G}^{-1}(t,p)$ with the string diagram $\Gamma$.
To do so, we must endow $\pi_{c_G}^{-1}(t,p)$ with a fatgraph structure and a marking of each output boundary cycle.
The cell $c_G$ is labeled by the fatgraph $G$ so $\pi_{c_G}^{-1}(t,p)$ has a canonical fatgraph structure for all $(t,p)$ in $c_G$. 
To promote this fatgraph structure to the structure of a string diagram, we must choose a marked point on each output boundary.
If in $G$ a point marking an output lies at a vertex, then we mark the corresponding vertex of $\pi_{c_G}^{-1}(t,p)$. 
If it lies in the interior of a directed edge, then we mark the corresponding directed edge of 
the 
output cycle of $\pi_{c_G}^{-1}(t,p)$ according to the $p_i$ coordinate of $x_\Gamma$.

Recall that input circles are marked by the $0$-cell of our model of $S^1$.

\begin{remark}
If $c_{G'}$ is a face of $c_{G}$ with inclusion map $i_{G',G}: c_{G'} \hookrightarrow c_G$, then for all $(t',p') \in c_{G'}$, $\pi_{c_{G'}}^{-1}(t',p')$ is a string diagram canonically isomorphic to the string diagram $\pi_{c_{G}}^{-1}(i(t',p'))$.
\end{remark}

\begin{definition}
Let $\tilde{i}_{G,G'}: \pi_{c_{G'}}^{-1}(c_{G'}) \to \pi_{c_G}^{-1}(i_{G',G}(c_{G'}))$ be the unique map such that 
\begin{itemize}
\item $\tilde{i}_{G',G} \circ \pi_{c_G} = \pi_{c_{G'}} \circ i_{G',G}$
\item $\tilde{i}_{G',G}$ restricts to the canonical isomorphism $\pi_{c_{G'}}^{-1}(t',p') \to \pi_{c_{G}}^{-1}(i(t',p'))$.
\end{itemize}
\end{definition}

\begin{definition}
Let $\{c_G\}$ be the collection of cells of $\sdbar$.
$$U\sdbar = \left( \bigsqcup_{c_G} U\sdbar(c_G) \right) / \sim$$
where if
 $c_{G'}$ a face of $c_G$,
 $(t',p') \in c_{G'}$,
 $y' \in \pi_{c_{G'}}^{-1}(t',p' ),  \tilde{i}_{G',G}(y') \in \pi_{c_G}^{-1}(i(t',p') )$,
 then
 $y' \sim \tilde{i}(y')$.
\end{definition}

\begin{remark}
Since the diagram
\begin{displaymath}
    \xymatrix{U\sdbar(c_{G'}) \ar[rr]^{\tilde{i}_{G',G}} \ar[d]_{\pi_{c_{G'}}}
    & & U\sdbar(c_{G}) \ar[d]^{\pi_{c_{G}}}      \\
               c_{G'} \ar[rr]^{i_{G', G}} \ar[rd]_{\Phi_{G'}}
             &  & c_G \ar[ld]^{\Phi_{G}} \\
             &  \sdbar &\\
              }
\end{displaymath}
commutes, the map $$\bigsqcup_G (\Phi_G \circ \pi_{c_G}): \bigsqcup_G U\sdbar(c_G) \to \sdbar$$ is well defined on $U\sdbar.$
\end{remark}

\begin{definition}
Let $\pi: U\sdbar \to \sdbar$ be the well-defined map $\bigsqcup_G \Phi_G \circ \pi_{c_G}$.
\end{definition}

\section{The String Topology Construction}\label{STC}
Let \(M\) be a compact, oriented, Riemannian manifold of dimension $d$. 
Let $g$, $k$, and $\ell$ be integers such that $g\ge 0$, $k>0$, and $\ell>0$ and such that:
\[ \chi := -(2-2g-k-\ell) \ge 1 . \]
The integer $\chi$ is minus the Euler characteristic of a Riemann surface of genus $g$ with $k+\ell$ boundary components.

Let $R$ be a commutative ring with 1.
Let \(\mathcal{C}_*(\sdbar)\) denote the cellular chain complex of the regular cell complex \(\sdbar\) with $R$ coefficients, and let $C_*$ denote the singular chain functor with $R$ coefficients.
We will suppress the indices $g$,$k$, and $\ell$ and write $\sdbar$ for $\sdbar$.
Let $LM$ denote the free loop space of $M$, and $LM^k$ and $LM^\ell$ denote the $k$ and $\ell$-fold Cartesian products of $LM$.

In this section we define a chain map
\[ \mathcal{ST}: \mathcal{C}_*(\sdbar) \otimes C_*(LM^k) \longrightarrow C_*(LM^\ell). \]
We will define the map on a generator of the free $R$-module
\[ \mathcal{C}_m(\sdbar) \otimes C_n(LM^k) \]
and then show that extending this map linearly produces a chain map. 
A generator of
\[ \mathcal{C}_m(\sdbar) \otimes C_n(LM^k) \]
is a pair $(c,\s)$, where $c$ is an $m$-cell of $\sdbar$ and $\s$ is a singular $n$-simplex of $LM^k$.
We will define a series of chain maps from $C_*(c \times \Delta^n)$ to $C_*(LM^\ell)$.
We will define $\mathcal{ST}(c,\s)$ to be the image of a certain chain in $C_{n+m}(c \times \Delta^n)$ under this series of chain maps.

\begin{prop}\label{prop adjoint}
The set 
\(\hbox{Maps}(\Delta^n,LM^k)\)
is a basis for the free $R$-module
$C_n(LM^k)$.
Let $\sqcup_{k} S^1$ denote the disjoint union of $k$ copies of $S^1$.
Then \(\hbox{Maps}(\Delta^n,LM^k)\) is isomorphic to the set
\[\M(\Delta^n \times \sqcup_{k} S^1, M).\]
\end{prop}
\begin{proof}
\begin{eqnarray*}
\hbox{Maps}\left(\Delta^n,LM^k\right) &=& \M\left(\Delta^n, \M(S^1,\,M)^k\right)\\
&=& \M\left(\Delta^n, \M(\sqcup_{k} S^1,\, M) \right)\\
&=& \M\left(\Delta^n \times \sqcup_{k} S^1,\, M\right)
\end{eqnarray*}
\end{proof}
We will abuse notation let $\s$ denote both a singular simplex of $LM^k$ and an element of $\M(\Delta^n \times \sqcup_{k} S^1, M)$.

Fix a generator $(c,\s)$ of $\mathcal{C}_m(\sdbar) \otimes C_n(LM^k)$, so that $c$ is an $m$-cell of $\sdbar$ and 
\[\s: \Delta^n \times \sqcup_{k} S^1 \longrightarrow M.\]

\subsection{The Thom class representative}
Let $\delta: M \to M \times M$ denote the diagonal map.
Then the $\chi$-fold Cartesian product of $\delta$ is a multi-diagonal map
\[\delta^\chi: M^\chi \to M^{2\chi}.\]
Let $D=\delta^\chi(M^\chi)$.
The manifold $M$ is Riemannian, and the metric $g$ induces a topological metric $d_g$ on $M$.
Let $d$ denote the metric on $M^{2\chi}$ defined by
\[ d((x_1, \ldots, x_{2\chi}),(y_1, \ldots, y_{2\chi})) = \mbox{max}\{ d_g(x_i,y_i)\,\, | \,\, 1 \le i \le 2\chi\}.\]
\begin{definition}For a positive real number $\e$,
\[ N_\e := \{ x \in M^{2\chi} \,\, | \,\, d(x,D) < \e\}.\]
\end{definition}\begin{prop}\label{prop epsilon}
For each point
\[ (x_1,y_1, \ldots , x_\chi, y_\chi) \in N_\e, \] 
there is a point 
\[(w_1,w_1, \ldots , w_\chi,w_\chi)\in D\] 
such that for each $i$
\[ d_g(x_i,w_i) < \e \;\;\; \mbox{ and } \;\;\; d_g(y_i,w_i) < \e.\]
\end{prop}
\begin{proof}
If 
\[ (x_1,y_1, \ldots , x_\chi, y_\chi) \in N_\e,\]
then
\[ d\left((x_1,y_1,\ldots , x_\chi, y_\chi), D\right) < \e.\]
Since $D$ is compact, there is a point
\[ (w_1, w_1, \ldots , w_\chi, w_\chi) \in D\]
which minimizes
\[ \left\{d\left((x_1,y_1, \ldots , x_\chi, y_\chi), w\right) \,\, | \,\, w \in D \right\}.\]
(The point $w$ need not be unique.)
In particular,
\[d\left((x_1,y_1, \ldots , x_\chi, y_\chi),(w_1, w_1, \ldots , w_\chi, w_\chi)\right) < \e \]
in the metric on $M^{2\chi}$.
Thus for each $i$,
\[ d_g(x_i,w_i) < \e \;\;\; \mbox{ and } \;\;\;  d_g(y_i,w_i) < \e.\]
\end{proof}

Let $[U]\in H^{\chi d}(M^{2\chi})$ denote the Thom class of the normal bundle of $D \subset M^{2\chi}$.
For small $\e$, $N_\e$ is diffeomorphic to the total space of this normal bundle.
In particular, if $\e$ is less than the injectivity radius of $(M,g)$, then $N_\e$ is diffeomorphic to the total space of the normal bundle. (See \cite{milnorstasheff}.)
\begin{definition}Let $\e$ denote the one half of the injectivity radius of $M$.
\end{definition}The Thom class $[U]$ can be represented by a cocycle 
\[ C^{\chi d}(M^{2\chi}, M^{2\chi} \setminus N_\frac{\e}{2}).\]
By the excision axiom, the inclusion
\[ 
(M^{2\chi} \setminus ( M^{2\chi} \setminus N_{\e}), (M^{2\chi} \setminus N_\frac{\e}{2}) \setminus ( M^{2\chi} \setminus N_\e)) \hookrightarrow (M^{2\chi}, M^{2\chi} \setminus N_\frac{\e}{2})\]
induces an isomorphism in cohomology.
Thus the Thom class $[U]$ can be represented by a cocycle in
\begin{eqnarray*}
 &&C^{\chi d}(M^{2\chi} \setminus ( M^{2\chi} \setminus N_{\e}), (M^{2\chi} \setminus N_\frac{\e}{2}) \setminus ( M^{2\chi} \setminus N_\e))\\
&=& C^{\chi d}(N_\e, N_\e \setminus N_\frac{\e}{2}) 
\end{eqnarray*}
We fix such a representative $U$.

\subsection{The evaluation map.}
We now define an evaluation map which will be used in the string topology construction.
A point $x_\Gamma \in \sdbar$ corresponds to a string diagram $\Gamma$.
Recall that such a string diagram is a CW-complex built by attaching $\chi$ copies of the interval $I$ to $k$ copies of the circle $S^1$.
Let $e_i(\Gamma)$ denote the $i$-th chord and let $\varphi_{e_i(\Gamma)}$ denote the $i$-th attaching map, so that 
\[ \varphi_{e_i(\Gamma)}: \{0,1\} \to \ks.\]
For each $1 \le i \le \chi$ and $j \in \{0,1\}$, we define a map
by the formula
\[
\begin{array}{lll}
 \tau_{i,j}:& \sdbar &\to \ks\\
 & x_\Gamma &\mapsto \varphi_{e_i(\Gamma)}(j).
\end{array}
\]
That is to say, $ \tau_{i,0}(x_\Gamma)$ is the initial vertex of the $i$-th chord of $\Gamma$, and $ \tau_{i,1}(x_\Gamma)$ is final vertex of the $i$-th chord of $\Gamma$.
Precomposing with $\tau_{i,j}$ gives a new map $\bar{\tau}_{i,j}$:
\begin{eqnarray*}
\bar{\tau}_{i,j}: \M(\ks, M) &\to& \M(\sdbar,M) \\
	f &\mapsto& f \tau_{i,j}.\\
\end{eqnarray*}
For every $n \ge 0$, each $\bar{\tau}_{i,j}$ induces a map
\[ ev(n,i,j): \M(\ks \times \Delta^n, M) \to \M(\sdbar \times \Delta^n ,M).\]
To be more explicit,
\begin{eqnarray*}
 ev(n,i,0)(\s)(x_\Gamma,t)  &=& \s\left(\varphi_{e_i(\Gamma)}(0),t\right)\\
&=& \s(\mbox{intital vertex of }i\mbox{-th chord of } \Gamma,t)
\end{eqnarray*}
and 
\begin{eqnarray*}
 ev(n,i,1)(\s)(x_\Gamma,t)  &=& \s\left(\varphi_{e_i(\Gamma)}(1),t\right)\\
&=& \s(\mbox{final vertex of }i\mbox{-th chord of } \Gamma,t).
\end{eqnarray*}
For each $n$, the product of the $2\chi$ maps $ev(n,i,j)$ is a map
\[ ev^n: \M(\ks \times \Delta^n, M) \to \M(\sdbar \times \Delta^n ,M^{2\chi}).\]
\begin{definition}\label{def barS}
Let $c$ be a subset of $\sdbar$ and let $\e>0$, and let 
\[ \s: \ksd \to M.\]
We define
\[ \tilde{S}_\e(c,\sigma) := \{ (x_\Gamma,t) \in c \times \Delta^n \,\, | \,\, ev^n(\s)(x_\Gamma,t) \in N_\e\}.\]
When $c$ and $\s$ are clear from context --- as is the case throughout this section, where they denote a fixed cell of $\sdbar$ and singular simplex of $M$ --- we will call this set $\tilde{S}_\e$.
\end{definition}\begin{prop}\label{prop within e}
Let $(x_\Gamma,t) \in \tilde{S}_\e(c,\sigma)$.
There exists a point \[w=(w_1,w_1, \ldots, w_\chi,w_\chi)\in D\]
such that for each $i$,
\[ d_g(\s(\varphi_{e_i(\Gamma)}(0),t),w_i) < \e \;\;\; \mbox{ and } \;\;\; d_g(\s(\varphi_{e_i(\Gamma)}(1),t),w_i) < \e.\]
\end{prop}
\begin{proof}
If $(x_\Gamma,t)$ lies in $ \tilde{S}_\e(c,\sigma)$, then by definition 
\[ ev^n(\s)(x_\Gamma,t) := \left(\s(\varphi_{e_1(\Gamma)}(0),t),\s(\varphi_{e_1(\Gamma)}(1),t), \,\, \ldots\,\, , \s(\varphi_{e_\chi(\Gamma)}(0),t),\s(\varphi_{e_\chi(\Gamma)}(1),t)\right)\]
lies in $ N_\e$.\
Thus by Proposition \ref{prop epsilon}, there exists a point $w\in D$ such that each $\s(\varphi_{e_i(\Gamma)}(0),t)$ and $\s(\varphi_{e_i(\Gamma)}(1),t)$ lie in a ball of radius epsilon in $M$ centered at $w_i$.
\end{proof}

\subsection{The fundamental chain of $c \times \Delta^n$.}\label{section c}
The topological space $c$ is a product of simplices, 
so the space $c \times \Delta^n$ is homeomorphic to $D^{n+m}$, and 
\[ H_{n+m}\left( c \times \Delta^n, \partial (c \times \Delta^n) \right)= R. \]
Let $j_\#$ denote the quotient map:
\[ j_\#: C_{n+m}(c \times \Delta^n) \to C_{n+m}( c \times \Delta^n, \partial (c \times \Delta^n)).\]
We would like to choose a cycle in $\mu \in C_{n+m}( c \times \Delta^n)$ such that $j_\#(\mu)$ represents a generator in \(H_{n+m}( c \times \Delta^n, \partial (c \times \Delta^n))\).
We now define the cycle explicitly.

The $m$-cell $c$ is a product of simplices, and so can be written as:
\[ c = c^1 \times \ldots \times c^p\]
where each factor $c^r$ is a simplex of dimension $j_r$, and 
\[ j_1 + \ldots + j_{p} = m.\]
The vertices of each simplex $c^r$ are ordered, so there is a unique ordered simplicial map
\[ \mu_{c^r}: \Delta^{j_r} \to c^r.\]
Moreover, this map is an element of $C_{j_r}(c^r)$.
Thus there is a singular chain given by the tensor product of simplicial maps:
\[ \mu_{c^1} \otimes \ldots \otimes \mu_{c^p} \in C_{j_1}(c^1) \otimes \ldots \otimes C_{j_{p}}(c^{p}).\]
Recall the theorem of Eilenberg and Zilber which states that the bifunctors
\[ \{X,Y\} \mapsto C_*(X) \otimes C_*(Y)\]
and 
\[ \{X,Y\} \mapsto C_*(X \times Y)\]
are naturally quasi-isomorphic~\cite{ez}.
Let $EZ$ denote the natural Eilenberg-Zilber quasi-isomorphism
\[ C_*(X) \otimes C_*(Y)\xlongrightarrow{EZ} C_*(X \times Y).\]
\begin{definition}\label{def muc}
Let $\mu_c$ denote the image of \( \mu_{c^1} \otimes \ldots \otimes \mu_{c^p}\) under the composition:
\[ C_{j_1}(c^{1}) \otimes \ldots \otimes C_{j_{p}}(c^{p})  \xlongrightarrow{EZ} C_m\left(c^{1} \times \ldots \times c^{p}\right) = C_m(c).\]
\end{definition}\begin{definition}\label{def 1n}
Let 
\[ 1_n: \Delta^n \to \Delta^n \]
denote the identity map, which is an element of $C_n(\Delta^n)$.
\end{definition}\begin{definition}\label{def mu}
Let $\mu_{c \times \Delta^n}$ denote the image of $\mu_c \otimes 1_n$ under the Eilenberg-Zilber map
\[ C_m (c) \otimes C_n(\Delta^n) \xlongrightarrow{EZ} C_{m+n}(c \times \Delta^n).\]
\end{definition}This chain $\mu_{c \times \Delta^n}$ is the desired chain in $C_{m+n}(c \times \Delta^n)$.

\subsection{The chain maps used to define $\mathcal{ST}$.}
We now define a series of chain maps, such that \(\mathcal{ST}(c,\s)\) is the image of $\mu_{c\times\Delta^n}$ under the composition of the maps in the series. 
The inclusion map of pairs
\[j: (c \times \Delta^n, \varnothing) \hookrightarrow (c \times \Delta^n, c \times \Delta^n \setminus \tilde{S}_\frac{\e}{2})\]
induces the quotient map $j_\#$:
\[C_*(c \times \Delta^n) \xlongrightarrow{j_\#} C_*(c \times \Delta^n, c \times \Delta^n \setminus \tilde{S}_\frac{\e}{2}).\] 
This map $j_\#$ is our first chain map.
The second map is excision:

\begin{align*}
C_*(c \times \Delta^n, &c \times \Delta^n \setminus \tilde{S}_\frac{\e}{2}) \\
&\xlongrightarrow{s} C_*\left(c \times \Delta^n \setminus \left(c \times \Delta^n \setminus \tilde{S}_\e\right), \left(c \times \Delta^n \setminus \tilde{S}_\frac{\e}{2}\right)\setminus \left(c \times \Delta^n\setminus \tilde{S}_\e\right)\right)\\
&= C_*( \tilde{S}_\e,  \tilde{S}_\e\setminus \tilde{S}_\frac{\e}{2}).
\end{align*}

More precisely, $s$ is a chain homotopy inverse to the quasi-isomorphism induced by the inclusion
\[\left(c \times \Delta^n \setminus \left(c \times \Delta^n \setminus \tilde{S}_\e\right), \left(c \times \Delta^n \setminus \tilde{S}_\frac{\e}{2}\right)\setminus \left(c \times \Delta^n\setminus \tilde{S}_\e\right)\right) \hookrightarrow (c \times \Delta^n, c \times \Delta^n \setminus \tilde{S}_\frac{\e}{2}).\]
See, for example, \cite[Proposition 2.21]{hatcher} for an explicit formula for $s$.

Next we must cap with the Thom class representative.
For a fixed $c$ and $\s$, the restriction of the evaluation is:
\[ ev^n(\s)|_{c\times \Delta^n}:c\times \Delta^n \to M^{2\chi}. \]
\begin{prop}\label{prop whut}
The further restriction of the evaluation map to $\tilde{S}_\e$  is a map of pairs:
\[ev^n(\s)|_{\tilde{S}_\e}: (\tilde{S}_\e,  \tilde{S}_\e\setminus \tilde{S}_\frac{\e}{2}) \to (N_\e,  N_\e\setminus N_\frac{\e}{2}).\] 
We will denote this restriction by $ev_c$.
\end{prop}
\begin{proof}
For $\tilde{S}_\e$ is defined to be the preimage of $N_\e$ under $ev^n(\s)$.
\end{proof}

We pullback the Thom cocycle $U$ by $ev_c$ to get a class 
\[ev_c^*(U)\in C^{\chi d}(\tilde{S}_\e,  \tilde{S}_\e\setminus \tilde{S}_\frac{\e}{2}).\]
The next map in our sequence is the cap product with this Thom class:
\[ C_*( \tilde{S}_\e,  \tilde{S}_\e\setminus \tilde{S}_\frac{\e}{2}) \xlongrightarrow{\cap ev_c^*(U)} C_{*-\chi d}(\tilde{S}_\e).\]

\subsection{Mapping string diagrams to $M$ using geodesics.}
The next map is the heart of the construction.
\begin{definition}\label{def mapssm}
Let $S \subset \sdbar$.
We define a space $\MM(S,M)$ as follows.
As a set,
\[ \MM(S,M) := \bigsqcup_{x_\Gamma \in S} \M(\Gamma, M).\]
Let 
\( p: USD \to SD\)
be the projection map.
The topology on $\MM(S,M)$ is generated by open sets of the following form:
\[ \left\{W \subset \M(p^{-1}(V),M)\,\, | \,\, V \mbox{ is an open set in }S\right\}.\]
\end{definition}
\begin{remark}A neighborhood of a point $f: \Gamma \to M$ in $\MM(S,M)$ is an open set
\[  W \subset \M(p^{-1}(V),M)\]
such that $V$ is a neighborhood of $x_\Gamma$ in $S$ and 
\[ F|_\Gamma = f\]
for some $F \in W$.
\end{remark}\begin{definition}\label{def pi}
Let 
\[\pi: c \times \Delta^n \to c\]
be the projection map.
Set
\[ S_\e := \pi(\tilde{S}_\e).\]
\end{definition}
We define a map of spaces
\[ \alpha^{in}: \tilde{S}_\e \to \MM(S_\e,M).\]
In order to do so, we must consider geodesics in $M$.
\begin{prop}\label{prop geos}
For each $1 \le i \le \chi$, there is a unique geodesic segment 
\[ \gamma_i: I \to M\]
such that:
\begin{eqnarray*}
\gamma_i(0) &=& \varphi_{e_i(\Gamma)}(0)\\
\gamma_i(1) &=& \varphi_{e_i(\Gamma)}(1).
\end{eqnarray*}
\end{prop}
Here $\varphi_{e_i(\Gamma)}(0)$ and $\varphi_{e_i(\Gamma)}(1)$ are the initial and final endpoints of the $i$-th chord of $\Gamma$.
This proposition says that there is a unique geodesic segment which starts at the initial vertex of the $i$-chord and ends and the final vertex of the $i$-th chord of $\Gamma$.
\begin{proof}
Since $(x_\Gamma,t)$ lies in $\tilde{S}_\e(c,\sigma)$, Proposition \ref{prop within e} says that there is a point 
\[ (w_1,w_1,\,\, \ldots \,\, ,w_\chi,w_\chi) \in D\]
such that $\s(\varphi_{e_i(\Gamma)}(0),t)$ and $\s(\varphi_{e_i(\Gamma)}(1),t)$ lie in a ball of radius $\e$ in $M$ centered at some $w_i$.
By the triangle inequality,
\begin{eqnarray*}
 d_g\left(\s(\varphi_{e_i(\Gamma)}(0),t), \s(\varphi_{e_i(\Gamma)}(1),t)\right) &\le& d_g\left(\s(\varphi_{e_i(\Gamma)}(0),t),w_i\right) + d_g\left(w_i,\s(\varphi_{e_i(\Gamma)}(1),t)\right)\\
&<& 2\e.
\end{eqnarray*}
Since $M$ is a complete Riemannian manifold with injectivity radius $2\e$, there is a unique geodesic segment
\[ \gamma_i: I \to M\]
such that 
\begin{eqnarray*}
\gamma_i(0) &=& \varphi_{e_i(\Gamma)}(0)\\
\gamma_i(1) &=& \varphi_{e_i(\Gamma)}(1).
\end{eqnarray*}
(This is a standard fact; see, for example, \cite[Theorem 14]{peterson}.)
\end{proof}
We define a map 
\[ \alpha^{in}: \tilde{S}_\e \to \MM(S_\e,M).\]
A point $(x_\Gamma,t)\in\tilde{S}_\e$ is sent to a map \[f_{x_\Gamma,t}:\Gamma \to M.\]
The graph $\Gamma$ is composed of metric, oriented, circles and chords.
Each circle is canonically identified with the standard circle $S^1$, and each chord $e_i$ is canonically identified with the standard interval $I$.
Thus to define a map $\Gamma \to M$, it suffices to define maps
\[ \ks \to M \]
and maps
\[ \sqcup_{i=1}^k e_i \to M\]
which agree at the attaching points of the chords.
\begin{definition}\label{def g}
We define
\begin{eqnarray*}
\tilde{S}_\e &\xlongrightarrow{\alpha^{in}}& \MM(S_\e,M)\\
(x_\Gamma,t)&\mapsto& f_{(x_\Gamma,t)}:\Gamma \to M.
\end{eqnarray*}
The map $ f_{(x_\Gamma,t)}$ is given by pasting together the following maps.
Since
\[ \s: \ks \times \Delta^n \to M,\]
for each $t\in \Delta^n$ we have a map
\[ \s_t: \ks \to M.\]
On input circles, we apply $\s_t$:
\[ f_{(x_\Gamma,t)}|_{\substack{\ks}} \equiv \s_t.\]
On the $i$-th chord, we follow the geodesic $\gamma^i$:
\[ f_{(x_\Gamma,t)}|_{\substack{e_i}} \equiv \gamma^i.
\]
Since $\gamma^i(0)= \s_t\left(\varphi_{e_i(\Gamma)}(0)\right)$ and $\gamma^i(1)= \s_t\left(\varphi_{e_i(\Gamma)}(1)\right)$, the maps agree on chord endpoints and paste together to give a well-defined map
\[ f_{(x_\Gamma,t)}: \Gamma \to M.\]
\end{definition}

The final map in the construction uses output boundary cycles to go from $\MM(S_e,M)$ to $LM^\ell$.
\begin{definition}We define a map
\[out: \MM(S_\e,M) \to LM^\ell\]
Recall from section \ref{sdbarUSD} that for any string diagram $\Gamma$, there is a map
\[b_\Gamma: \els \to \Gamma\]
which maps the $i$-th circle onto the $i$-th output boundary cycle of $\Gamma$ by reversing orientation.
Given a map 
\[f:\Gamma \to M,\] 
define
\[ out(f): \els \xlongrightarrow{b_\Gamma}\Gamma \xlongrightarrow{f} M.\]
\end{definition}

\subsection{The map $\mathcal{ST}$.}
\begin{definition}\label{def stc}
Consider the following composition of maps:
\begin{eqnarray*}
C_*(c \times \Delta^n) &\xlongrightarrow{j_\#}& C_*\left(c \times \Delta^n, c \times \Delta^n \setminus \tilde{S}_\frac{\e}{2}\right) \\
&\xlongrightarrow{s}& C_*\left(c \times \Delta^n \setminus \left(c \times \Delta^n \setminus \tilde{S}_\e\right), \left(c \times \Delta^n \setminus \tilde{S}_\frac{\e}{2}\right)\setminus \left(c \times \Delta^n\setminus \tilde{S}_\e\right)\right)\\
&=& C_*\left( \tilde{S}_\e,  \tilde{S}_\e\setminus \tilde{S}_\frac{\e}{2}\right)\\
&\xlongrightarrow{\cap ev_c^*(U)}& C_{*-\chi d}\left(\tilde{S}_\e\right)\\
&\xlongrightarrow{\alpha^{in}_\#}& C_{*-\chi d}\left(\MM(S_e,M)\right)\\
&\xlongrightarrow{out_\#}& C_{*-\chi d}\left(LM^\ell\right).
\end{eqnarray*}
We denote this composition $\st_{(c,\s)}$.
We define $\mathcal{ST}(c,\s)$ to be the image of $\mu_{c \times \Delta^n}$ under this composition:
\begin{eqnarray*}
\mathcal{ST}(c,\s) 
&:=& \st_{(c,\s)}(c,\s)\\
&:=& out_\# \circ \alpha^{in}_\# \left( s\circ j_\# (\mu_{c \times \Delta^n})\cap ev_c^*U\right).
\end{eqnarray*}
\end{definition}
\section{$\mathcal{ST}$ is a chain map.}
In this section we check that $\mathcal{ST}$ is a chain map.
The graded module 
\[\mbox{Hom}\left(\mathcal{C}_*(\sdbar) \otimes C_*(LM^k), C_*(LM^\ell)\right)\]
is a chain complex with differential
\[ d_{\mbox{Hom}} f := \partial f - (-1)^{\mbox{deg}(f)} f d_\otimes.\]
This choice of signs for the differential ensures that a 0-cycle is a chain map and that a 0-boundary is a null-homotopic chain map.
\begin{theorem}\label{thm chain}
The map
\[ \mathcal{ST}: \mathcal{C}_*(\sdbar) \otimes C_*(LM^k) \longrightarrow C_*(LM^\ell) \]
satisfies
\[ d_{\mbox{Hom}}(\mathcal{ST}) = 0.\]
\end{theorem}
\begin{remark}Let $\partial$ denote the singular differential, let $d$ denote the cellular differential in $\mathcal{C}_*(\sdbar)$, and let $d_\otimes$ denote the differential in $\mathcal{C}_*(\sdbar) \otimes C_*(LM^k)$.
Then the statement is:
\begin{eqnarray*}
 \partial \mathcal{ST}(c,\s) &=& (-1)^{\chi d} \mathcal{ST}d_\otimes(c,\s)\\
&=&(-1)^{\chi d} \left( \mathcal{ST} (dc,\s) + (-1)^m \mathcal{ST} (c,\partial \s)\right).
\end{eqnarray*}
\end{remark}The idea of the proof is as follows.
The chain $\mathcal{ST}(c,\s)$ is defined to be $\st_{(c,\s)}(\mu_{c\times \Delta^n})$, where $\st_{(c,\s)}$ is a composition of chain maps and $\mu_{c\times \Delta^n}$ is a chain in $C_{n+m}(c\times \Delta^n)$.
More precisely, $\mu_{c\times \Delta^n}$ represents a generator of $H_{n+m}(c\times \Delta^n, \partial(c\times \Delta^n))$.
Thus $\mu_{c\times \Delta^n}$ should be thought of as a ``fundamental chain'' of $c\times \Delta^n$.We show that in the appropriate sense, ``the boundary of fundamental chain is the fundamental chain of the boundary''.
The precise statement is Lemma \ref{nate lem3}.

The other issue is that the map $\st_{(c,\s)}$ depends on $c$ and $\s$.
The boundary of the chain $\mu_{c\times \Delta^n}$ has terms coming corresponding to the faces of $c$ and of $\Delta^n$.
We must check that applying $\st_{(c,\s)}$ to a term in the boundary of the chain $\mu_{c\times \Delta^n}$ coming from a face $\partial c$ of $c$ gives the same result as applying $\st_{(\partial c,\s)}$ to $\mu_{\partial c \times \Delta^n}$.
Similarly, we must check that applying $\st_{(c,\s)}$ to a term in the boundary of the chain $\mu_{c\times \Delta^n}$ coming from a face $\partial \s$ of $\s$ gives the same result as applying $\st_{(c, \partial \s)}$ to $\mu_{c \times \partial \Delta^n}$.
The precise statements are Lemmas \ref{nate lem2} and \ref{nate lem1}.

Now we proceed with the proof.
First we fix some notation.
Let $c$ denote a fixed $m$-cell of $\sdbar$ and let $\sigma$ denote a fixed singular $n$-simplex of $LM^k$.
As discussed in section \ref{section c}, the cell $c$ is a product of simplices
\[ c = c^1 \times \, \ldots \, \times c^p,\]
where $c^r$ is a simplex of dimension $j_r$.
Since the dimension of $c$ is $m$, we have
\[ \sum_{j=1}^p j_r = m.\]
Moreover, the vertices of each simplex factor are ordered.
Let $\partial_s c^r$ denote the face of $c^r$ given by omitting the $s$-th vertex, and let $ \partial_{rs} c $ denote the product
\[c^1 \times \, \ldots \, \times \partial_s c^r \times \, \ldots\, \times c^p.\]
There is a unique ordered simplicial map
\[ \partial_{rs}: \partial_{rs} c = \left( c^1 \times \, \ldots \, \times \partial_s c^r \times \, \ldots\, \times c^p\right) \to \left( c_1 \times \, \ldots \,\times c^r \times \,\ldots \, \times c^p\right)=c.\]
Crossing with the identity on $\Delta^n$ gives a map:
\[ \partial_{rs} \times 1_n: \partial_{rs} c \times \Delta^n \to c \times \Delta^n.\]
Let 
\[ \partial_i: \Delta^{n-1} \to \Delta^{n} \]
denote the unique map of ordered simplices which omits the $i$-th vertex.
Crossing with the identity on $c$ give a map:
\[ 1_c \times \partial_i: c \times \Delta^{n-1} \to c \times \Delta^n.\]

Using these maps, we able to state the three lemmas from which the theorem follows.
\begin{lemma}\label{nate lem3}
Recall from Definitions~\ref{def muc},~\ref{def 1n}, and ~\ref{def mu} the chain
\[ \mu_{c \times \Delta^n} := EZ(\mu_{c^1} \otimes \ldots \otimes \mu_{c^p} \otimes 1_n) \in C_{n+m}(c \times \Delta^n).\]
Let 
\[ \epsilon(r,s) = s + \sum_{u=1}^{r-1} j_u .\]
Then we have:
\[\partial \mu_{c \times \Delta^n} = 
\sum_{r=1}^p \sum_{s=1}^{j_r} (-1)^{\epsilon(r,s)}(\partial_{rs}\times 1_n)_\# \left(\mu_{\partial_{rs} c \times \Delta^n}\right)
+ (-1)^m \sum_{i=1}^n (-1)^i (1_c \times \partial_i)_\# (\mu_{c\times\Delta^{n-1}}).\]
\end{lemma}
\begin{proof}
Since the Eilenberg-Zilber map is a natural transformation, the following diagram commutes.

\(
\xymatrix {
C_*(c^1)\otimes \ldots \otimes C_*(\partial_s c^r) \otimes \ldots \otimes C_*(c^p)\otimes C_*(\Delta^n) 
\ar[d]^-{(1_{c^r})_\# \otimes \ldots \otimes (\partial_s)_\# \otimes \ldots \otimes (1_{c^p})_\# \otimes (1_n)_\#} 
\ar[r]^-{EZ}
& C_*(c^1 \times \, \ldots \, \times \partial_s c^r \times \, \ldots\, \times c^p \times \Delta^n) \ar[d]^-{(\partial_{rs}\times 1_n)_\#}\\
 C_*(c^1)\otimes \ldots \otimes C_*(c^r)\otimes \ldots \otimes  C_*(c^p)\otimes C_*(\Delta^n) 
\ar[r]^-{EZ}
 &C_*(c_1 \times \, \ldots \,\times c^r \times \,\ldots \, \times c^p \times \Delta^n).
}
\)
Thus
\begin{eqnarray*}
(\partial_{rs}\times 1_n)_\# \left(\mu_{\partial_{rs} c \times \Delta^n}\right) &:=&
(\partial_{rs}\times 1_n)_\# EZ \left( \mu_{c^1} \olo \mu_{ \partial_s c^r} \olo  \mu_{c^p} \otimes 1_n\right)\\
&=& EZ \left( (1_{c^r})_\#\mu_{c^1} \olo  (\partial_s)_\# \mu_{ \partial_s c^r} \olo(1_{c^p})_\#\mu_{c^p} \otimes (1_n)_\# 1_n\right).
\end{eqnarray*}
Since the identity map on spaces induces the identity map on chains, we have:
\begin{eqnarray*}
(\partial_{rs}\times 1_n)_\# \left(\mu_{\partial_{rs} c \times \Delta^n}\right) &=&
EZ \left(\mu_{c^1} \olo (\partial_s)_\# \mu_{ \partial_s c^r} \olo\mu_{c^p} \otimes 1_n \right).
\end{eqnarray*}
Now, $(\partial_s)_\# \mu_{ \partial_s c^r}$ is the map
\begin{equation}\label{one}
\Delta^{j_r-1} \xlongrightarrow{\mu_{ \partial_s c^r}}\partial_s c^r  \xlongrightarrow{\partial_s} c^r.
\end{equation}
Here $\mu_{ \partial_s c^r}$ is the canonical simplicial map of $j_r$-dimensional simplices with ordered vertices, and $\partial_s$ is the simplicial map from a $j_r-1$ simplex to a $j_r$ simplex with omits the $s$-th vertex.
Conversely, the term $\partial_s \mu_{c^r}$ which appears in the simplicial boundary of $\mu_{c^r}$ is the map
\begin{equation}\label{two}
 \Delta^{j_r-1} \xlongrightarrow{\partial_s} \Delta^{j_r} \xlongrightarrow{\mu_{ \partial_s c^r}} c^r.
\end{equation}
Now, the compositions (\ref{one}) and (\ref{two}) are the same simplicial map, so we have
\begin{equation}\label{three}
(\partial_{rs}\times 1_n)_\# \left(\mu_{\partial_{rs} c \times \Delta^n}\right) = 
EZ \left(\mu_{c^1} \olo \partial_s \mu_{c^r} \olo\mu_{c^p} \otimes 1_n \right)
\end{equation}
in $C_{n+m-1}(c \times \Delta^n)$.

A completely analogous argument, again using the naturality of the Eilenberg-Zilber map, shows that:
\begin{equation}\label{four}
(1_c \times \partial_i)_\# (\mu_{c\times\Delta^{n-1}}) = EZ \left(\mu_{c^1}\olo \mu_{c^p} \otimes \partial_i 1_n\right)
\end{equation}
in $C_{n+m-1}(c \times \Delta^n)$.

We compute:
\begin{eqnarray*}
\partial \mu_{c \times \Delta^n} 
&:=& \partial EZ \left( \mu_{c^1} \otimes \ldots \otimes \mu_{c^p}\otimes 1_n\right) \\
&=& EZ \partial\left(\mu_{c^1} \otimes \ldots \otimes \mu_{c^p}\otimes 1_n\right)\\
&=& EZ \Big( \partial\left(\mu_{c^1} \otimes \ldots \otimes \mu_{c^p}\right)\otimes 1_n 
+ (-1)^m \left( \mu_{c^1} \otimes \ldots \otimes \mu_{c^p}\right)\otimes \partial 1_n\Big)\\
&=& EZ \Big( \sum_{r=1}^p (-1)^{\sum_{u=1}^{r-1} j_u}\mu_{c^1} \otimes \ldots \otimes \partial \mu_{c^r}\otimes \ldots \otimes \mu_{c^p}\otimes 1_n 
\\
&\;& \hspace{4cm}
+ (-1)^m \mu_{c^1} \otimes \ldots \otimes \mu_{c^p}\otimes \partial 1_n \Big)\\
&=& EZ \Big( \sum_{r=1}^p \sum_{s=1}^{j_r} (-1)^{\epsilon(r,s)} \mu_{c^1} \otimes \ldots \otimes \partial_s \mu_{c^r} \otimes \ldots \otimes \mu_{c^p} 
\\
&\;&\hspace{4cm}+ 
(-1)^m \sum_{i=1}^n (-1)^i \mu_{c^1} \otimes \ldots \otimes \mu_{c^p}\otimes \partial_i 1^n \Big)\\
&=& \sum_{r=1}^p \sum_{s=1}^{j_r} (-1)^{\epsilon(r,s)} EZ \left( \mu_{c^1} \otimes \ldots \otimes \partial_s \mu_{c^r} \otimes \ldots \otimes \mu_{c^p}\right) 
\\
&\;&\hspace{4cm}
+ (-1)^m \sum_{i=1}^n (-1)^i EZ \left( \mu_{c^1} \otimes \ldots \otimes \mu_{c^p}\otimes \partial_i 1^n \right)\\
&=& \sum_{r=1}^p \sum_{s=1}^{j_r} (-1)^{\epsilon(r,s)} (\partial_{rs}\times 1_n)_\# (\mu_{\partial_{rs}c \times \Delta^n}) + (-1)^m \sum_{i=1}^n (-1)^i(1_c \times \partial_i)_\# (\mu_{c\times\Delta^{n-1}}).
\end{eqnarray*}
The final equality follows from (\ref{three}) and (\ref{four}).
\end{proof}

We proceed with the next lemma.
\begin{lemma}\label{nate lem2}
The following diagram commutes.\\
\begin{center}
\(
\xymatrix{
C_{n+m-1}(c \times \Delta^n) \ar[r]^{\st_{(c,\s)}}& C_{n+m-1-\chi d}(LM^\ell)\\
C_{n+m-1}(\partial_{rs} c \times \Delta^{n}) \ar[u]^{( \partial_{rs}\times 1_n)_\#} \ar[ur]_{\st_{(\partial_{rs} c,\s)}}
}
\)
\end{center}
\end{lemma}
In words, the lemma says that applying $\st_{(\partial_{rs} c,\s)}$ gives the same result as first including $\partial_{rs}c \times \Delta^{n}$ as a face of $c \times \Delta^n$ and then applying $\st_{(c,\s)}$.
\begin{proof}
Recall that $\st_{(c,\s)}$ and $\st_{(\partial_{rs} c,\s)}$ are compositions of the several maps of Definition~\ref{def stc}.
We show the above diagram commutes by showing that various maps induced by $(\partial_{rs}\times 1_n)_\#$ commute with each of the maps of Definition~\ref{def stc}.
More precisely, we will show that the following diagram commutes:
\begin{equation}\label{dia foo}
\xymatrix{
C_{n+m-1}(c \times \Delta^n) \ar[r]^-{j_\#} & \bullet \ar[r]^-s & \bullet \ar[rr]^{\cap ev_c*(U)} && \bullet  \ar[r]^-{\alpha^{in}_\#} & \bullet \ar[r]^-{out_\#}& C_{n+m-1-\chi d}(LM^\ell). \\
C_{n+m-1}(\partial_{rs} c \times \Delta^n) \ar[u]_-{(\partial_{rs}\times 1_n)_\#} \ar[r]^-{j_\#} & \bullet\ar[u]_-{(\partial_{rs}\times 1_n)_\#} \ar[r]^-s & \bullet\ar[u]_-{(\partial_{rs}\times 1_n)_\#} \ar[rr]^{\cap ev_{\partial_{rs} c}*(U)} && \bullet \ar[u]_-{(\partial_{rs}\times 1_n)_\#} \ar[r]^-{\alpha^{in}_\#} & \ar[u]_-{(\partial_{rs})_\#}\bullet \ar[ur]_-{out_\#}
}
\end{equation}
The names of some of the entries in the diagram are suppressed to make the diagram easier to read.
The four squares and the triangle in the above diagram are Diagrams (\ref{dia two}), (\ref{dia three}), (\ref{dia four}), (\ref{dia eight}), and (\ref{dia ten}) below.
Since 
\[\st_{(c,\s)}(x) := out_\# (\alpha_{in})_\# \big( (sj_\# x) \cap ev_c^*U\big),\]
the commutativity of (\ref{dia foo}) implies the lemma.

We now proceed with the proof.
Recall from Definition~\ref{def barS} the space $\tilde{S}_\e(c,\s)$ associated to a pair $(c,\s)$:
\[ \tilde{S}_\frac{\e}{2}(c,\s) = (ev^n(\s)|_{c\times\Delta^n})^{-1}(N_\frac{\e}{2}), \]
where 
\[ ev^n(\s): \sdbar \times \Delta^n \to M^{2\chi}.\]
The following diagram commutes:
\begin{equation}\label{dia one}
\xymatrix{
c \times \Delta^n \ar[r]^{ev^n(\s)|_{c\times\Delta^n}} & M^{2\chi}.\\
\partial_{rs} c \times \Delta^n \ar[u]^{\partial_{rs} \times 1_n} \ar[ur]_{ev^n(\s)|_{\partial_{rs} c\times \Delta^n}}
}
\end{equation}
Thus, $\partial_{rs}\times 1_n$ maps $\partial_{rs} c \times \Delta^n \setminus \tilde{S}_\frac{\e}{2}(\partial_{rs} c,\s)$ into $ c \times \Delta^n \setminus \tilde{S}_\frac{\e}{2}(c,\s)$.
Therefore the map $\partial_{rs}\times 1_n$ gives a well-defined map of pairs
\[ \partial_{rs}\times 1_n: \left(\partial_{rs} c \times \Delta^{n}, \partial_{rs} c \times \Delta^n \setminus \tilde{S}_\frac{\e}{2}(\partial_{rs} c,\s)\right) \to \left(c \times \Delta^n, c \times \Delta^n \setminus \tilde{S}_\frac{\e}{2}(c,\s)\right).\]
Thus the following diagram, which is the leftmost square in (\ref{dia foo}), commutes
\begin{equation}\label{dia two}
\xymatrix{
C_{n+m-1}(c \times \Delta^n) \ar[r]^-{j_\#} 
& C_{n+m-1}\left(c \times \Delta^n, c \times \Delta^n \setminus \tilde{S}_\frac{\e}{2}(c,\s)\right)\\
C_{n+m-1} (\partial_{rs} c \times \Delta^{n}) \ar[r]^-{j_\#} \ar[u]^{( \partial_{rs}\times 1_n)_\#}
& C_{n+m-1}\left(\partial_{rs} c \times \Delta^{n}, \partial_{rs} c \times \Delta^n \setminus \tilde{S}_\frac{\e}{2}(\partial_{rs} c,\s)\right). \ar[u]^{( \partial_{rs}\times 1_n)_\#}
}
\end{equation}
Using the commutativity of (\ref{dia one}), we see that $\partial_{rs}\times 1_n$ maps $\tilde{S}_\e(\partial_{rs} c,\s)$ into $\tilde{S}_\e(c,\s)$ and maps $\tilde{S}_\e(\partial_{rs} c,\s) \setminus \tilde{S}_\frac{\e}{2}(\partial_{rs} c,\s)$ into $\tilde{S}_\e(c,\s)\setminus \tilde{S}_\frac{\e}{2}(c,\s)$.
Thus, $\partial_{rs}\times 1_n$ restricts to a well-defined map of pairs:
\[ \partial_{rs}\times 1_n: \left(\tilde{S}_\e(\partial_{rs} c,\s), \tilde{S}_\e(\partial_{rs} c,\s) \setminus \tilde{S}_\frac{\e}{2}(\partial_{rs} c,\s)\right) \to \left(\tilde{S}_\e(c,\s), \tilde{S}_\e(c,\s)\setminus \tilde{S}_\frac{\e}{2}(c,\s)\right).\]

Consider second square in (\ref{dia foo}).
\begin{equation}\label{dia three}
\xymatrix{
C_{n+m-1}\left(c \times \Delta^n, c \times \Delta^n \setminus \tilde{S}_\frac{\e}{2}(c,\s)\right)
\ar[r]^-{s} 
& C_{n+m-1}\left( \tilde{S}_\e(c,\s), \tilde{S}_\e(c,\s) \setminus \tilde{S}_\frac{\e}{2}(c,\s)\right)\\
C_{n+m-1}\left(\partial_{rs} c \times \Delta^{n}, \partial_{rs} c \times \Delta^n \setminus \tilde{S}_\frac{\e}{2}(\partial_{rs} c,\s)\right) \ar[u]^{( \partial_{rs}\times 1_n)_\#} \ar[r]^-{s}
& C_{n+m-1}\left(\tilde{S}_\e(\partial_{rs} c,\s), \tilde{S}_\e(\partial_{rs} c,\s) \setminus \tilde{S}_\frac{\e}{2}(\partial_{rs} c,\s)\right). \ar[u]^{( \partial_{rs}\times 1_n)_\#}
}
\end{equation}
The horizontal arrows are the chain level excision maps.
If $\tau$ is a singular simplex of $\partial_{rs} c \times \Delta^{n}$, then $s(\tau)$ given by performing two operations.
First subdivide $\tau$ into smaller simplices that lie entirely in $\partial_{rs} c \times \Delta^{n} \setminus \tilde{S}_\frac{\e}{2}(\partial_{rs} c,\s)$ or entirely in $\tilde{S}_\e(\partial_{rs} c,\s)$, and then discard all simplices of the first type.
Similarly, $s(( \partial_{rs}\times 1_n)_\#\tau)$ is given by first subdividing $(\partial_{rs}\times 1_n)_\#\tau$ into simplices that lie entirely in $c  \times \Delta^{n} \setminus \tilde{S}_\frac{\e}{2}(c,\s)$ or entirely in $\tilde{S}_\e(c,\s)$ then discarding simplices of the first type.
See the proof of~\cite[Proposition 2.21]{hatcher} for explicit formulas for $s$.
Observe that
\[ \left(c \times \Delta^n \setminus \tilde{S}_\frac{\e}{2}(c,\s)\right) \bigcap \left(\partial_{rs} c \times \Delta^{n}\right) = \partial_{rs} c \times \Delta^{n} \setminus \tilde{S}_\frac{\e}{2}(\partial_{rs} c,\s)
\]
and 
\[ \tilde{S}_\e(c,\s) \bigcap \partial_{rs} c \times \Delta^{n} = \tilde{S}_\e(\partial_{rs}c,\s). \]
Thus for a singular simplex $\tau$ of $\partial_{rs} c \times \Delta^{n}$, 
\[s( \partial_{rs}\times 1_n)_\#\tau = ( \partial_{rs}\times 1_n)s\tau\]
and so Diagram~\ref{dia three} commutes.

The next diagram we consider is the third square in (\ref{dia foo}):
\begin{equation}\label{dia four}
\xymatrix{
 C_{n+m-1}\left( \tilde{S}_\e(c,\s), \tilde{S}_\e(c,\s) \setminus \tilde{S}_\frac{\e}{2}(c,\s)\right) \ar[rr]^-{\cap ev_c^*(U)}
& &  C_{n+m-1-\chi d}\left(\tilde{S}_\e(c,\s)\right)\\
 C_{n+m-1}\left(\tilde{S}_\e(\partial_{rs} c,\s), \tilde{S}_\e(\partial_{rs} c,\s) \setminus \tilde{S}_\frac{\e}{2}(\partial_{rs} c,\s)\right) \ar[u]^{( \partial_{rs}\times 1_n)_\#}
\ar[rr]^-{\cap ev_{\partial_{rs}c}^*(U)}
& &  C_{n+m-1-\chi d}\left(\tilde{S}_\e(\partial_{rs}c,\s)\right). 
\ar[u]^{( \partial_{rs}\times 1_n)_\#}
}
\end{equation}
Recall from Proposition~\ref{prop whut} that $ev_c$ is abbreviated notation for the restriction of 
\[ ev^n(\s): \sdbar \to M^{2\chi}\]
to $\tilde{S}_\e(c,\s)$, and that $ev_c$ is a map of pairs:
\[ev_c := ev^n(\s)|_{\tilde{S}_\e(c,\s)}: \left(\tilde{S}_\e(c,\s),  \tilde{S}_\e(c,\s)\setminus \tilde{S}_\frac{\e}{2}(c,\s)\right) \to (N_\e,  N_\e\setminus N_\frac{\e}{2}).\] 
Since $ev_{\partial_{rs} c}$ is simply the further restriction of $ev^n(\s)$ to $\tilde{S}_\e(\partial_{rs} c,\s)$, the following diagram commutes:
\begin{equation}\label{dia five}
\xymatrix{
\left(\tilde{S}_\e(c,\s),  \tilde{S}_\e(c,\s)\setminus \tilde{S}_\frac{\e}{2}(c,\s)\right)
\ar[r]^-{ev_c} 
& (N_\e,  N_\e\setminus N_\frac{\e}{2}) \subset M^\chi.\\
\left(\tilde{S}_\e(\partial_{rs} c,\s) \setminus \tilde{S}_\frac{\e}{2}(\partial_{rs} c,\s)\right)
\ar[u]^{\partial_{rs} \times 1_n} \ar[ur]_{ev_{\partial_{rs} c}}
}
\end{equation}
Thus for a chain 
\[x \in  C_{n+m-1}\left(\tilde{S}_\e(\partial_{rs} c,\s), \tilde{S}_\e(\partial_{rs} c,\s) \setminus \tilde{S}_\frac{\e}{2}(\partial_{rs} c,\s)\right),\]
we have the following.
\begin{eqnarray*}
(\partial_{rs}\times 1_n)_\# \left( x \cap ev_{\partial_{rs}c}^*(U)\right)
&=& (\partial_{rs}\times 1_n)_\# \big( x \cap (\partial_{rs}\times 1_n)^* (ev_c)^* U\big)\\
&=& \big( (\partial_{rs}\times 1_n)_\#\, x\big) \cap (ev_c)^* U.
\end{eqnarray*}
Thus Diagram (\ref{dia four}) commutes.

We now consider the commutativity of the boundary map $\partial_{rs}$ with the map $\alpha^{in}$.
Recall that $\partial_{rs}$ denotes the inclusion
\[\partial_{rs}: \partial_{rs} c = c_1 \times \ldots \times \partial_s c^r \times \ldots \times c^p \hookrightarrow c_1 \times \ldots \times c^p = c\]
of the face $\partial_{rs} c$ into the cell $c$.
The map $\partial_{rs} c$ induces an inclusion
\[ \partial_{rs}: \tilde{S}_\e(\partial_{rs} c,\s)  \hookrightarrow \tilde{S}_\e(c,\s).\]
Let $\pi$ denote the projection maps
\[ \pi:c \times \Delta^n \to c\]
and 
\[\pi: \partial_{rs} c \times \Delta^n \to\partial_{rs} c.\]
Then the following diagram commutes:
\begin{equation}\label{dia six}
\xymatrix{
\partial_{rs} c \times \Delta^n \ar[d]^\pi \ar[rr]^{\partial_{rs} \times 1_n} & & c \times \Delta^n \ar[d]^\pi\\
\partial_{rs} c \ar[rr]^{\partial_{rs}} & & c.
}
\end{equation}
Recall from Definition~\ref{def pi} that 
\[S_\e(c,\s) := \pi \tilde{S}_\e(c,\s).\]
Thus the inclusion $\partial_{rs}$ induces an inclusion
\[ \partial_{rs}: S_\e(\partial_{rs} c,\s)  \hookrightarrow S_\e(c,\s).\]

Recall from Definition~\ref{def mapssm} the space 
\[ \MM(S,M) := \bigsqcup_{x_\Gamma \in S} \M(\Gamma, M).\]
The map $\partial_{rs}$ induces an inclusion
\[ \partial_{rs}: \MM(S_\e(\partial_{rs}c,\s),M) \hookrightarrow \MM(S_\e(c,\s),M).\]
A point in $\MM(S_\e(\partial_{rs}c,\s)$ is a map
\[f: \Gamma \to M\]
where $x_\Gamma \in S_\e(\partial_{rs}c,\s)$.
Then the image $\partial_{rs}(f)$ of $f$ under the inclusion $\partial_{rs}$ is simply the same map $f$.

Consider the following diagram:
\begin{equation}\label{dia seven}
\xymatrix{
 \tilde{S}_\e(c,\s) \ar[r]^-{\alpha^{in}}
& \MM(S_\e(c,\s),M)\\
 \tilde{S}_\e(\partial_{rs}c,\s) 
\ar[u]^{ \partial_{rs}\times 1_n} \ar[r]^-{\alpha^{in}} & \MM(S_\e(\partial_{rs}c,\s),M). \ar[u]^{\partial_{rs}}
}
\end{equation}
Let $(x_\Gamma,t)\in \tilde{S}_\e(\partial_{rs}c,\s)$.
Recall from Definition~\ref{def g} that 
$\alpha^{in}(x_\Gamma,t)$ is a map 
\[f_{(x_\Gamma,t)}: \Gamma \to M.\]
The definition of $f_{(x_\Gamma,t)}$ depends on $x_\Gamma$, $\Gamma$, and $\sigma$, but not the ambient cell $\partial_{rs} c$.
Thus $\alpha^{in}(\partial_{rs}\times 1_n)(x_\Gamma,t)$ and $\partial_{rs}\alpha^{in}(x_\Gamma,t)$ are the same map
\[f_{(x_\Gamma,t)}: \Gamma \to M.\]
Therefore Diagram (\ref{dia seven}) commutes.

Applying the singular chain functor to Diagram (\ref{dia seven}) gives the fourth square in Diagram (\ref{dia foo}):
\begin{equation}\label{dia eight}
\xymatrix{
 C_{n+m-1-\chi d}\left(\tilde{S}_\e(c,\s)\right) \ar[r]^-{\alpha^{in}_\#}
& C_{n+m-1-\chi d}\left(\MM(S_\e(c,\s),M)\right)\\
C_{n+m-1-\chi d} \left(\tilde{S}_\e(\partial_{rs}c,\s)\right) 
\ar[u]^{ (\partial_{rs}\times 1_n)_\#} \ar[r]^-{\alpha^{in}_\#} & C_{n+m-1-\chi d}\left(\MM(S_\e(\partial_{rs}c,\s),M)\right). \ar[u]^{(\partial_{rs})_\#}
}
\end{equation}

Finally, consider the following diagram:
\begin{equation}\label{dia nine}
\xymatrix{
\MM(S_\e(c,\s),M) \ar[r]^-{out}& LM^\ell.\\
\MM(S_\e(\partial_{rs}c,\s),M) \ar[u]^{\partial_{rs}} \ar[ur]_-{out}.
}
\end{equation}
A point in $\MM(S_\e(\partial_{rs}c,\s),M) $ is a map 
\[f:\Gamma \to M\]
where $x_\Gamma$ is a point in $S_\e(\partial_{rs}c,s)$.
The map $out_\#(f)$ is the composition
\[ \bigsqcup_\ell S^1 \xlongrightarrow{b}  \Gamma \xlongrightarrow{f} M,\]
where $b$ is the output boundary cycle map of $\Gamma$.
Since $out_\#(f)$ depends only on the string diagram $\Gamma$, $out_\#\partial_{rs}(f)$ is precisely the same map.
Thus Diagram (\ref{dia nine}) commutes.

Taking chains, we have to following commutative diagram, which is the triangle in Diagram (\ref{dia foo}).
\begin{equation}\label{dia ten}
\xymatrix{
C_{n+m-1-\chi d}\left(\MM(S_\e(c,\s),M)\right) \ar[r]^-{out_\#}& C_{n+m-1-\chi d}\left(LM^\ell\right).\\
C_{n+m-1-\chi d}\left(\MM(S_\e(\partial_{rs}c,\s),M)\right) \ar[u]^{(\partial_{rs})_\#} \ar[ur]_-{out_\#}.
}
\end{equation}

Combining Diagrams (\ref{dia two}), (\ref{dia three}), (\ref{dia four}), (\ref{dia eight}), and (\ref{dia ten}) shows that Diagram (\ref{dia foo}) commutes and gives the lemma.
\end{proof}

We proceed with the next lemma.
\begin{lemma}\label{nate lem1}
The following diagram commutes.\\
\begin{center}
\(
\xymatrix{
C_{n+m-1}(c \times \Delta^n) \ar[r]^{\st_{(c,\s)}}& C_{n+m-1-\chi d}(LM^\ell)\\
C_{n+m-1}(c \times \Delta^{n-1}) \ar[u]^{(1_c \times \partial_i)_\#} \ar[ur]_{\st_{(c,\partial_i \s)}}
}
\)
\end{center}
\end{lemma}
Recall that $\st_{(c,\s)}$ is the series of chain maps rising in the string topology construction for $(c,\s)$.
In words, the lemma says that applying $\st_{(c,\partial_i\s)}$ gives the same result as first including $c \times \Delta^{n-1}$ as a face of $c \times \Delta^n$ and then applying $\st_{(c,\s)}$.
\begin{proof}
The proof is quite similar to the proof of Lemma \ref{nate lem2}, so we will omit some of the details.
However, the roles of $c$ and $\s$ in the construction are not identical, so some slightly different arguments are needed.

Once again, we will show the diagram commutes by showing that a more complicated diagram commutes:
\begin{equation}\label{dia bar}
\xymatrix{
C_{n+m-1}(c \times \Delta^n) \ar[r]^-{j_\#} & \bullet \ar[r]^-s & \bullet \ar[rr]^{\cap ev_c*(U)} && \bullet  \ar[r]^-{\alpha^{in}_\#} & \bullet \ar[r]^-{out_\#}& C_{n+m-1-\chi d}(LM^\ell). \\
C_{n+m-1}(c \times \Delta^{n-1}) \ar[u]_-{(1_c \times \partial_i)_\#} \ar[r]^-{j_\#} & \bullet\ar[u]_-{(1_c \times \partial_i)_\#} \ar[r]^-s & \bullet\ar[u]_-{(1_c \times \partial_i)_\#} \ar[rr]^{\cap ev_{\partial{rs} c}*(U)} && \bullet \ar[u]_-{(1_c \times \partial_i)_\#} \ar[r]^-{\alpha^{in}_\#} & \ar[u]_-{i_\#}\bullet \ar[ur]_-{out_\#}
}
\end{equation}
The names of some of the entries in the diagram are suppressed to make the diagram easier to read.

The inclusion
\[ \partial_i: \Delta^{n-1} \to \Delta^n\]
induces the vertical maps in the following commutative diagram:
\begin{equation}\label{dia 2two}
\xymatrix{
\M(\ks \times \Delta^n, M) \ar[r]^-{ev^n} \ar[d]^-{\partial_i}& \M(\sdbar \times \Delta^n, M^{2 \chi}) \ar[d]^-{\partial_i}\\
\M(\ks \times \Delta^{n-1}, M) \ar[r]^-{ev^{n-1}} & \M(\sdbar \times \Delta^{n-1}, M^{2 \chi}).
}
\end{equation}
The commutativity of (\ref{dia 2two}) implies that the following diagram commutes:
\begin{equation}\label{dia 2three}
\xymatrix{
c \times \Delta^n \ar[rr]^-{ev^n(\s)|_{c \times \Delta^n}} & & M^{2\chi}.\\
c \times \Delta^{n-1} \ar[urr]_-{ev^{n-1}(\partial_i\s)|_{c \times \Delta^{n-1}}} \ar[u]^-{1_c \times \partial_i}
}
\end{equation}
Recall that
\[ \tilde{S}_\e(c,\s) := (ev^n(\s)|_{c \times \Delta^n})^{-1}(N_\e(c,\s)) \subset M^{2 \chi}.\]
Thus by the commutativity of (\ref{dia 2three}), $1_c \times \partial_i$ induces well-defined maps of pairs:
\[ 1_c \times \partial_i: \left(c \times \Delta^{n-1},  c \times \Delta^{n-1} \setminus \tilde{S}_\frac{\e}{2}(c,\partial_{i} \s)\right) \to \left(c \times \Delta^{n}, c \times \Delta^{n} \setminus \tilde{S}_\frac{\e}{2}(c,\s)\right)\]
and
\[ 1_c \times \partial_i:  \left(\tilde{S}_\e(c,\partial_i\s), \tilde{S}_\e(c,\partial_{i} \s) \setminus \tilde{S}_\frac{\e}{2}(c,\partial_{i} \s)\right) \to \left(\tilde{S}_\e(c,\s), \tilde{S}_\e(c,\s)\setminus \tilde{S}_\frac{\e}{2}(c,\s)\right).\]
These maps form right edge of the first and second squares, respectively, of Diagram (\ref{dia bar}).

The first two squares in Diagram (\ref{dia bar}) are:
\begin{equation}\label{dia 2one}
\xymatrix{
C_{n+m-1}(c \times \Delta^n) \ar[r]^-{j_\#} & C_{n+m-1}\left(c \times \Delta^n, c \times \Delta^n \setminus \tilde{S}_\frac{\e}{2}(c,\s)\right)\\
C_{n+m-1}(c \times \Delta^{n-1}) \ar[u]_-{(1_c \times \partial_i)_\#} \ar[r]^-{j_\#} & C_{n+m-1}\left(c \times \Delta^{n-1},  c \times \Delta^{n-1} \setminus \tilde{S}_\frac{\e}{2}(c,\partial_{i} \s)\right) \ar[u]_-{(1_c \times \partial_i)_\#}
}
\end{equation}
and
\begin{equation}\label{dia 2four}
\xymatrix{
C_{n+m-1}\left(c \times \Delta^n, c \times \Delta^n \setminus \tilde{S}_\frac{\e}{2}(c,\s) \right)\ar[r]^-s & C_{n+m-1} \left(\tilde{S}_\e(c,\s), \tilde{S}_\e(c,\s)\setminus \tilde{S}_\frac{\e}{2}(c,\s)\right) \\
C_{n+m-1}\left(c \times \Delta^{n-1},  c \times \Delta^{n-1} \setminus \tilde{S}_\frac{\e}{2}(c,\partial_{i} \s)\right) \ar[u]_-{(1_c \times \partial_i)_\#} \ar[r]^-s & C_{n+m-1}\left(\tilde{S}_\e(c,\partial_i\s), \tilde{S}_\e(c,\partial_{i} \s) \setminus \tilde{S}_\frac{\e}{2}(c,\partial_{i} \s)\right) \ar[u]_-{(1_c \times \partial_i)_\#} 
}
\end{equation}
An argument completely analogous to the one given in the proof of Lemma~\ref{nate lem2} for Diagrams (\ref{dia two}) and (\ref{dia three}) shows that Diagrams (\ref{dia 2one}) and (\ref{dia 2four}) commute.

We proceed to the third square in (\ref{dia bar}).
The maps in diagram (\ref{dia 2three}) restrict to give the following commutative diagram:
\begin{equation}\label{dia 2five}
\xymatrix{
\left(\tilde{S}_\e(c,\s), \tilde{S}_\e(c,\s)\setminus \tilde{S}_\frac{\e}{2}(c,\s)\right)  \ar[rr]^-{ev^n(\s)|_{\tilde{S}_\e(c,\s) }} & & \big(N_\e, N_\e \setminus N_\frac{\e}{2}\big) \subset M^{2\chi}.\\
\left(\tilde{S}_\e(c,\partial_{i} \s), \tilde{S}_\e(c,\partial_{i} \s) \setminus \tilde{S}_\frac{\e}{2}(c,\partial_{i} \s)\right) \ar[urr]_-{ev^{n-1}(\partial_i\s)|_{\tilde{S}_\e(c,\partial_{i} \s) }} \ar[u]^-{1_c \times \partial_i}
}
\end{equation}
Thus for a chain 
\[ x \in  C_{n+m-1}\left(\tilde{S}_\e(c,\partial_i\s), \tilde{S}_\e(c,\partial_{i} \s) \setminus \tilde{S}_\frac{\e}{2}(c,\partial_{i} \s)\right) \]
we have the following.
\begin{eqnarray*}
(1_c \times \partial_i)_\# \left( x \cap \left( ev^{n-1}(\partial_i\s)|_{\tilde{S}_\e(c,\partial_{i} \s) }\right)^* U \right)
&=& (1_c \times \partial_i)_\# \big( x \cap (\partial_{rs}\times 1_n)^* \big(ev^n(\s)|_{\tilde{S}_\e(c,\s)}\big)^* U\big)\\
&=& (1_c \times \partial_i)_\#\, x\cap \big(ev^n(\s)|_{\tilde{S}_\e(c,\s)}\big)^* U.
\end{eqnarray*}
This computation says precisely that the following diagram, which is the third square of (\ref{dia bar}), commutes.
\begin{equation}\label{dia 2six}
\xymatrix{
 C_{n+m-1}\left( \tilde{S}_\e(c,\s), \tilde{S}_\e(c,\s) \setminus \tilde{S}_\frac{\e}{2}(c,\s)\right) \ar[rrr]^-{\cap \big(ev^n(\s)|_{\tilde{S}_\e(c,\s)}\big)^* U}
& & &  C_{n+m-1-\chi d}\left(\tilde{S}_\e(c,\s)\right)\\
 C_{n+m-1}\left(\tilde{S}_\e(c,\partial_i s), \tilde{S}_\e(c,\partial_i s) \setminus \tilde{S}_\frac{\e}{2}(c,\partial_i s)\right) \ar[u]^{( 1_c \times \partial_i )_\#}
\ar[rrr]^-{\cap \left((ev^{n-1}(\partial_i\s)|_{\tilde{S}_\e(c,\partial_{i} \s) }\right)^* U}
& & &  C_{n+m-1-\chi d}\left(\tilde{S}_\e(c,\partial_i s)\right). 
\ar[u]^{( 1_c \times \partial_i )_\#}
}
\end{equation}

We now consider the fourth square in Diagram(\ref{dia bar}).
Recall from Definition~\ref{def barS} and Proposition~\ref{prop whut} that
\[ \tilde{S}_\e(c,\sigma) := \{ (x_\Gamma, t) \in c \times \Delta^n \,\, | \,\, ev^n(\sigma)(x_\Gamma,t) \in N_\e\}\subset c \times \Delta^n\]
and 
\[ S_\e(c,\sigma) := \pi \tilde{S}_\e(c,\sigma) \subset c.\]
Thus 
\[ S_\e(c,\sigma) = \{ x_\Gamma \in c \,\, | \,\, ev^n(\sigma)(x_\Gamma,t) \in N_\e \mbox{ for some } t \in \Delta^n\}\]
and similarly
\[ S_\e(c,\partial_i \sigma) = \{ x_\Gamma \in c \,\, | \,\, ev^{n-1}(\partial_i \sigma)(x_\Gamma,t) \in N_\e \mbox{ for some } t \in \Delta^{n-1}\}.\]
By the commutativity of Diagram~\ref{dia 2three}, $S_\e(c,\partial_i \sigma)$ is a subset of $ S_\e(c,\sigma)$.
Let 
\[i:S_\e(c,\partial_i \sigma) \hookrightarrow  S_\e(c,\sigma)\]
denote the inclusion.
This inclusion induces an inclusion
\[i: \MM(S_\e(c,\partial_i \sigma),M) \hookrightarrow \MM(S_\e(c,\sigma),M).\]

We claim that the following diagram commutes:
\begin{equation}\label{dia 2seven}
\xymatrix{
 \tilde{S}_\e(c,\s) \ar[r]^-{\alpha^{in}}
& \MM(S_\e(c,\s),M)\\
 \tilde{S}_\e(c,\partial_i s) 
\ar[u]^{ 1_c \times \partial_i} \ar[r]^-{\alpha^{in}} & \MM(S_\e(c,\partial_i s),M). \ar[u]^{i}
}
\end{equation}
Let $(x_\Gamma,t) \in \tilde{S}_\e(c,\partial_i s)$.
Then $\alpha^{in}(1_c \times \partial_i)(x_\Gamma,t)$ is a map 
\[f_{(1_c \times \partial_i)(x_\Gamma,t)}:\Gamma \to M\]
and $\alpha^{in}(x_\Gamma,t)$ is a map
\[f_{(x_\Gamma,t)}:\Gamma \to M.\]
For a point $\theta$ on an input circle of $\Gamma$,
\[f_{(1_c \times \partial_i)(x_\Gamma,t)}(\theta) = \sigma(t,\theta).\]
and 
\[f_{(x_\Gamma,t)}(\theta) = \partial_i \sigma(t,\theta).\]
Since $t$ is a point in $\partial_i \Delta^{n}$,
\[ \sigma(t,\theta) = \partial_i  \sigma(t,\theta).\]
Thus, the maps $f_{(1_c \times \partial_i)(x_\Gamma,t)}$ and $f_{(x_\Gamma,t)}(\theta)$ agree on the input circles of $\Gamma$.
The behavior of the maps $f_{(1_c \times \partial_i)(x_\Gamma,t)}$ and $f_{(x_\Gamma,t)}$ and on the chords of $\Gamma$ is given by the geodesic construction of Proposition~\ref{prop geos}.
Since $f_{(1_c \times \partial_i)(x_\Gamma,t)}$ and $f_{(x_\Gamma,t)}$ agree on input circles, they are agree on chord endpoints.
The construction of Proposition~\ref{prop geos} depends only on the images of the chords endpoints in $M$, so the maps $f_{(1_c \times \partial_i)(x_\Gamma,t)}$ and $f_{(x_\Gamma,t)}$ are determined by their restriction to input circles.
Thus
\[ f_{(1_c \times \partial_i)(x_\Gamma,t)} \equiv f_{(x_\Gamma,t)}.\]
In particular,
\[ i( f_{(x_\Gamma,t)}) = f_{(1_c \times \partial_i)(x_\Gamma,t)}\]
so Diagram~\ref{dia 2seven} commutes.

Applying the singular chain functor to Diagram (\ref{dia 2seven}) gives fourth square in Diagram(\ref{dia bar}):
\begin{equation}\label{dia 2eight}
\xymatrix{
 C_{n+m-1-\chi d}\left(\tilde{S}_\e(c,\s)\right) \ar[r]^-{\alpha^{in}_\#}
& C_{n+m-1-\chi d}\left(\MM(S_\e(c,\s),M)\right)\\
C_{n+m-1-\chi d} \left(\tilde{S}_\e(c,\partial_i s)\right) 
\ar[u]^{ (1_c \times \partial_i)_\#} \ar[r]^-{\alpha^{in}_\#} & C_{n+m-1-\chi d}\left(\MM(S_\e(c,\partial_i s),M)\right). \ar[u]^{i_\#}
}
\end{equation}

The triangle in Diagram (\ref{dia bar}) is
\begin{equation}\label{dia 2ten}
\xymatrix{
C_{n+m-1-\chi d}\left(\MM(S_\e(c,\s),M)\right) \ar[r]^-{out_\#}& C_{n+m-1-\chi d}\left(LM^\ell\right).\\
C_{n+m-1-\chi d}\left(\MM(S_\e(c,\partial_i s),M)\right) \ar[u]^{{i}_\#} \ar[ur]_-{out_\#}
}
\end{equation}
An argument completely analogous to the one given for Diagram (\ref{dia ten})  in the proof of Lemma~\ref{nate lem2} shows that (\ref{dia 2ten}) commutes.

Combining Diagrams (\ref{dia 2one}), (\ref{dia 2four}), (\ref{dia 2six}), (\ref{dia 2eight}), and (\ref{dia 2ten}) shows that Diagram (\ref{dia bar}) commutes, and the lemma follows.
\end{proof}
\begin{proof}[Proof of Theorem \ref{thm chain}.]
Fix a generator $(c,\s)$ of $\mathcal{C}_*(\sdbar) \otimes C_*(LM^k) $.
We compute:
\begin{eqnarray*}
\partial \mathcal{ST}(c,\s) &:=& \partial \st_{(c,\s)}(c,\s)\\
&:=& \partial out_\# \alpha^{in}_\# \left( sj_\# (\mu_{c \times \Delta^n})\cap ev_c^*U\right).
\end{eqnarray*}
Since all the maps in the above composition are chain maps, we have:
\begin{eqnarray*}
\partial \mathcal{ST}(c,\s) &=& out_\# \alpha^{in}_\# \partial\left( sj_\# (\mu_{c \times \Delta^n})\cap ev_c^*U\right)\\
&=& out_\# \alpha^{in}_\# (-1)^{\chi d} \left( \left( sj_\# \partial (\mu_{c \times \Delta^n})\cap ev_c^*U\right) -  sj_\# (\mu_{c \times \Delta^n})\cap ev_c^*\delta  U\right).
\end{eqnarray*}
Since $U$ is a cocycle, $\delta U = 0$, and we have:
\begin{eqnarray*}
\partial \mathcal{ST}(c,\s) &=&out_\# \alpha^{in}_\# (-1)^{\chi d} \left( \left( sj_\# \partial (\mu_{c \times \Delta^n})\cap ev_c^*U\right) -  sj_\# (\mu_{c \times \Delta^n})\cap ev_c^*\delta  U\right)\\
&=& out_\# \alpha^{in}_\# \left( (-1)^{\chi d} \left( sj_\# \partial (\mu_{c \times \Delta^n})\cap ev_c^*U\right) \right)\\
&=& (-1)^{\chi d}\st_{(c,\s)} (\partial \mu_{c \times \Delta^n}).
\end{eqnarray*}
By Lemma \ref{nate lem3},
\begin{eqnarray*}
\partial \mathcal{ST}(c,\s) &=& (-1)^{\chi d}\st_{(c,\s)} (\partial \mu_{c \times \Delta^n})\\
&=& (-1)^{\chi d}\st_{(c,\s)} \Big(\sum_{r=1}^p \sum_{s=1}^{j_r} (-1)^{\epsilon(r,s)}(\partial_{rs}\times 1_n)_\# \left(\mu_{\partial_{rs} c \times \Delta^n}\right)
\\
&\;& \hspace{4cm}
+ (-1)^m \sum_{i=1}^n (-1)^i (1_c \times \partial_i)_\# \left(\mu_{c\times\Delta^{n-1}}\right)\Big)\\
&=&  (-1)^{\chi d}\Big(\sum_{r=1}^p \sum_{s=1}^{j_r} (-1)^{\epsilon(r,s)}\st_{(c,\s)}(\partial_{rs}\times 1_n)_\# \left(\mu_{\partial_{rs} c \times \Delta^n}\right)  
\\
&\;& \hspace{4cm}
+(-1)^m \sum_{i=1}^n (-1)^i\st_{(c,\s)}(1_c \times \partial_i)_\# \left(\mu_{c\times\Delta^{n-1}}\right)\Big).
\end{eqnarray*}
Using Lemmas \ref{nate lem2} and \ref{nate lem1}, we continue:
\begin{eqnarray*}
\partial \mathcal{ST}(c,\s) 
&=& (-1)^{\chi d} \left(\sum_{r=1}^p \sum_{s=1}^{j_r} (-1)^{\epsilon(r,s)} \st_{(\partial_{rs}c, \s)}\mu_{\partial_{rs}c\times\Delta^n} + (-1)^m \sum_{i=1}^n (-1)^i \st_{(c,\partial_i s)} \left(\mu_{c\times\Delta^{n-1}}\right)\right)\\
&=& (-1)^{\chi d}\left(\sum_{r=1}^p \sum_{s=1}^{j_r} (-1)^{\epsilon(r,s)} \mathcal{ST}(\partial_rs c, \s) + (-1)^m \sum_{i=1}^n (-1)^i \mathcal{ST}(c,\partial_i \s) \right)\\
&=& (-1)^{\chi d}\left( \mathcal{ST} (dc,\s) + (-1)^m \mathcal{ST}(c,\partial \s) \right)\\
&=& (-1)^{\chi d} \mathcal{ST} d_\otimes (c,\s).
\end{eqnarray*}
\end{proof}

The proof of Theorem~\ref{thm chain} actually shows something slightly more general.
Let $W$ be any cochain in $M^{2\chi}$ supported near $D$.
That is say, let
\[ W \in C^{w}\left( N_\e, N_\e - N_\frac{\e}{2}\right).\]
Then we can substitute $W$ for $U$ in Definition~\ref{def stc} to get a degree $-w$ map
\[ \mathcal{ST}_W: \mathcal{C}_*(\sdbar) \otimes C_*(LM^k) \to C_{*-w}(LM^\ell).\]
\begin{prop}\label{prop w}
The boundary of $\mathcal{ST}_W$ in
\[\mbox{Hom}_{-w}\left(\mathcal{C}_*(\sdbar) \otimes C_*(LM^k),\, C_{*}(LM^\ell)\right) \]
is as follows:
\[ d_{\mbox{Hom}}\mathcal{ST}_W =(-1)^{w+1} \mathcal{ST}_{\delta W}.\]
\end{prop}
\begin{proof}
Let $(c,\s)$ be a generator of $\mathcal{C}_m(\sdbar) \otimes C_n(LM^k)$.
The above computation in the proof of Theorem~\ref{thm chain} shows that:
\[ \partial \mathcal{ST}_W(c,\s) = (-1)^w \mathcal{ST} d_\otimes (c,\s) - (-1)^{w} \mathcal{ST}_{\delta W}(c,\s).\]
Thus we have:
\begin{eqnarray*}
d_{\mbox{Hom}}\mathcal{ST}_W(c,\s) &:=& \partial \mathcal{ST}_W(c,\s) - (-1)^w \mathcal{ST}_Wd_\otimes (c,\s)\\
&=& (-1)^w \mathcal{ST} d_\otimes (c,\s) - (-1)^{w} \mathcal{ST}_{\delta W}(c,\s) - (-1)^w \mathcal{ST}_W d_\otimes (c,\s)\\
&=& (-1)^{w+1} \mathcal{ST}_{\delta W}(c,\s).
\end{eqnarray*}
\end{proof}
Now we can establish how our map $\mathcal{ST}$ changes if we replace $U$ by a different representative of the Thom class.
\begin{cor}\label{cor new thom}
Let be $U$ and $U'$ be two representatives of the Thom cohomology class in
\[H^{\chi d}\left(N_\e, N_\e \setminus N_\frac{e}{2}\right).\]
Then $\mathcal{ST}_U$ and $\mathcal{ST}_U'$ differ by a boundary in
\[\mbox{Hom}_{-w}\left(\mathcal{C}_*(\sdbar) \otimes C_*(LM^k),\, C_{*}(LM^\ell)\right). \]
\end{cor}
\begin{proof}
Since $U$ and $U'$ represent the same cohomology class,
there is a cochain 
\[ W \in C^{\chi d}\left(N_\e, N_\e \setminus N_\frac{e}{2}\right)\]
such that
\[ \delta W = (-1)^{w+1}(U - U'). \]
We compute
\begin{eqnarray*}
d_{\mbox{Hom}}(\mathcal{ST}_W ) &=& (-1)^{w+1} \mathcal{ST}_{\delta W}\\
&=& (-1)^{w+1} \mathcal{ST}_{(-1)^{w+1}(U - U')} \\
&=& \mathcal{ST}_U - \mathcal{ST}_{U'}.
\end{eqnarray*}
\end{proof}
\begin{remark}\label{rem ez}
We have defined a chain map
\[ \mathcal{ST}: \mathcal{C}_*(\sdbar) \otimes C_*(LM^k) \longrightarrow C_*(LM^\ell). \]
Using the Eilenberg-Zilber functor, we can define a new map
\[ \widetilde{\mathcal{ST}}: \mathcal{C}_*(\sdbar) \otimes C_*(LM)^{\otimes k} \to C_*(LM)^{\otimes \ell}.\]
To be explicit, 
 $\widetilde{\mathcal{ST}} $ is the composition
 \[ \mathcal{C}_*(\sdbar) \otimes C_*(LM)^{\otimes k} \xrightarrow{ 1 \otimes EZ^{-1} } \mathcal{C}_*(\sdbar) \otimes C_*(LM^k) \xrightarrow{\mathcal{ST}} C_*(LM^\ell) \xrightarrow{EZ} C_*(LM)^{\otimes \ell}.\]
\end{remark}


\begin{example}\label{loopproduct}
The chain-level loop product is the map
$$\widetilde{\mathcal{ST}}(c_\Gamma, -): C_*(LM) \otimes C_*(LM) \to C_*(LM)$$
where $c_\Gamma$ is the $0$-cell of $\sdbar(0,2,1)$ corresponding to the string diagram $\Gamma$ of type $(0,2,1)$ with the following properties:
\begin{enumerate}
\item
For the chord $e_\Gamma$, $\varphi_{e_\Gamma}(0)$ coincides with the marked point on the first input and $\varphi_{e_\Gamma}(1)$ coincides with the marked point on the second input.
\item
The point marking the output coincides with the vertex $v_1$ on input 1 between the directed edge $\vec{e}$ with target $v_1$ and the directed edge corresponding to the first input circle.
\end{enumerate}
See Figure 5.

In Corollary \ref{corcor} we will see that the chain-level loop product induces a commutative algebra structure on $H_*(LM)$ which agrees with the structure induced by the Chas-Sullivan loop product.
\end{example}

 \begin{figure}[h]
\centering
\includegraphics{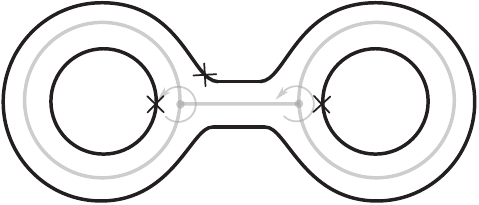}
\label{fatgraph3}
\caption{The string diagram giving the chain-level loop product.}
\end{figure}

\section{Induced operations on homology}

Sullivan chord diagrams were introduced by Cohen-Godin in \cite{CG} to define string topology operations on the homology of the loop space.
In this section we show that the string topology construction defined in Section \ref{STC} recovers those defined in \cite{CG}.

\begin{definition}{\cite{CG}}
A Sullivan chord diagram of type $\gkl$ is a fat graph of type $(g, k + \ell)$ that consists of a disjoint union of $k$ disjoint circles together with the disjoint union of connected trees whose endpoints lie on the circles. The cyclic orderings of the edges at the vertices must be such that each of the $k$ disjoint circles is a boundary cycle. These $k$ circles are referred to as the incoming boundary cycles and the other $\ell$ boundary cycles are referred to as outgoing boundary cycles. Edges of the trees are referred to as ghost edges.
\end{definition}

Again, we fix $(g, k, \ell)$ for the remainder of this section.

\begin{definition}
Let $Sull$ be the space of marked metric Sullivan chord diagrams of type $\gkl$.
\end{definition}

\begin{remark}
Let $\Gamma \in SD$, then $\Gamma$ is a marked metric Sullivan chord diagram. In particular, $\sdbar \cap Sull = SD$.
\end{remark}

In \cite{CG}, Cohen and Godin define operations 
$$\mu_\Gamma:h_*(LM)^{\otimes k} \to h_*(LM)^{\otimes \ell}$$ for $\Gamma$ a marked metric Sullivan chord diagram and $h_*$ any homology theory supporting an orientation of $M$. Let 
$$Maps(S(\Gamma), M) = \{f: \Gamma \to M \; | \; f \textrm{ is constant on each ghost edge}\}$$ and let $\rho_{in}$ and $\rho_{out}$ be restrictions of such maps to inputs and outputs respectively:
$$LM^k \stackrel{\; \rho_{in}}{\longleftarrow} Maps(S(\Gamma), M) \stackrel{\rho_{out}}{\longrightarrow} LM^\ell.$$
Cohen and Godin show that  $\rho_{in}$ is a finite codimension embedding and apply a Thom collapse to obtain an umkehr map on homology:
$$h_*(LM^k) \stackrel{\; (\rho_{in})_!}{\longrightarrow} h_{* - \chi d}(Maps(S(\Gamma), M)) \stackrel{(\rho_{out})_*}{\longrightarrow} h_{* - \chi d}(LM^\ell).$$
To be more explicit, let  $Maps(S(\Gamma), M)^{\nu(\Gamma)}$ denote the Thom space of the normal bundle
of $\rho_{ {in}_!} \left(Maps(S(\Gamma),M)\right)$ inside $LM^k$.
Let 
\[\tau: LM^k \to Maps(S(\Gamma), M)^{\nu(\Gamma)}\]
denote the Thom collapse map, and let 
\[t:  h_*(Maps(S(\Gamma), M)^{\nu(\Gamma)}) \to h_{*-\chi d}(Maps(S(\Gamma),M))\]
denote the Thom isomorphism.
Then  $(\rho_{in})_! = t \circ \tau_*$.

\begin{definition}{\cite{CG}} Let $\Gamma$ be a marked metric Sullivan chord diagram of type $\gkl$. 
Then $$\mu_\Gamma = (\rho_{out})_* \circ (\rho_{in})_!: h_*(LM^k)  \to h_{* - \chi d}(LM^\ell).$$
\end{definition}

We use the notation $\mu_\Gamma$  for this operation as in \cite{CG}; this should not be confused with the notation $\mu_{c \times \Delta^n}$, introduced in Section \ref{STC}, for the fundamental chain of the space $c \times \Delta^n$.

Now we consider the string topology construction of Section \ref{STC} for a fixed string diagram $\Gamma$, and compare it to the string topology operations $\mu_\Gamma$ when $h_* = H_*$, that is, singular homology with field coefficients.

\begin{definition}
Let $\Gamma$ be a string diagram of type $\gkl$. If $x_\Gamma$ is not a 0-cell in the cell decomposition of $\sdbar$, then it is in the interior of a higher dimensional cell.
Subdivide this cell by taking its barycentric subdivision using $x_\Gamma$ as the barycenter.
This sudivision gives a new cell decomposition of $\gkl$ for which $x_\Gamma$ is a 0-cell.
Let $c_\Gamma$ denote this 0-cell, and define a map
$$\lambda_\Gamma: C_*(LM^k)  \to C_{* - \chi d}(LM^\ell)$$
by 
\[ \lambda_\Gamma(\s) = \mathcal{ST}(c_\Gamma, \s).\]
\end{definition}

\begin{prop}
The map $\lambda_\Gamma$ satisfies
\[ \partial \lambda_\Gamma = (-1)^{\chi d} \lambda_\Gamma \partial.\]
\end{prop}
\begin{proof}
Since $d c_\Gamma = 0$, the statement follows immediately from Theorem \ref{thm chain}.
\end{proof}

\begin{theorem}\label{CG}
Let $x_\Gamma \in SD$, let the coefficient ring $R$ be a field and let $(\lambda_\Gamma)_*$ be the map induced on homology by $\lambda_\Gamma$. Then
$$H_*(LM)^{\otimes k} \cong H_*(LM^k)  \stackrel{(\lambda_\Gamma)_*}{\longrightarrow} H_{* - \chi d}(LM^\ell) \cong H_{* - \chi d}(LM)^{\otimes l}$$
is equal to $\mu_\Gamma$.

\end{theorem}

\begin{proof}

Recall from Section \ref{STC} the definition of the evaluation map $$ev_{c_\Gamma, \sigma}: \Delta^n \to M^{2\chi}$$ for $\sigma: \Delta^n \to LM^k$. 
Recall from \cite{CG} the definition of the evaluation map $e_\Gamma: LM^k \to M^{|v(\Gamma)|}$.
Here 
$$V(\Gamma)=\{v^1_1, \dots, v^1_{n_1}, v^2_1,  \dots, v^2_{n_2}, \dots, v^k_1, \dots, v^k_{n_k}\}$$ is the set of vertices of $\Gamma$, and the evaluation map is given by
$$e_\Gamma(\gamma_1, \gamma_2, \dots, \gamma_k) = (\gamma_1(v^1_1), \dots, \gamma_1(v^1_{n_1}), \gamma_2(v^2_1),  \dots, \gamma(v^2_{n_2}), \dots, \gamma_k(v^k_1), \dots, \gamma_k(v^k_{n_k})).$$

If $\Gamma$ has chord endpoints coinciding on input circles, then $|v(\Gamma)| < 2\chi$. 
Let $v^i_j$ have multiplicity $m^i_j$. Define the map $$i_\Gamma: M^{|V(\Gamma)|} \to M^{2\chi}$$
 by first repeating the coordinate $\gamma_i(v^i_j)$ a total of $(m^i_j -2)$ times and then permuting the coordinates according to the ordering of the chords of $\Gamma$.

Let $\mathscr{F}(g, n)$ be the space of marked metric fatgraphs with genus $g$ and $n$ boundary components. Let  $$S: SD \to \mathscr{F}(g, k+ \ell)$$ be the map obtained by collapsing each chord to a point and let $V(S(\Gamma))$ be the set of vertices of the fatgraph corresponding to $S(x_\Gamma)$. Let 
$$i'_\Gamma: M^{|V(S(\Gamma))|} \to M^{\chi}$$
be defined by repeating coordinates and permuting so that the lower square in the following diagram commutes.
The upper square in the diagram is Cohen and Godin's pullback square.  The full diagram commutes. We are particularly interested in the triangle relating Cohen and Godin's evaluation map to the evaluation map defined in section \ref{STC} for $\sigma: \Delta^n \to LM^k$ a generator of $C_n(LM^k)$.

\begin{displaymath}
    \xymatrix{Maps(S(\Gamma), M) \ar[rr]^{\rho_{in}} \ar[d]_{e_\Gamma}
    & & LM^k \ar[d]^{e_\Gamma}      \\
               M^{|V(S(\Gamma))|} \ar[rr]^{\Delta_\Gamma} \ar[d]
             &  & M^{|v(\Gamma)|} \ar[d] ^{i_\Gamma}& & \Delta^n \ar[ull]_{\sigma} \ar[dll]^{ev_{c_\Gamma}}\\
M^{\chi} \ar[rr]^{\delta^\chi} 
             &  & M^{2\chi} \\
                 }
\end{displaymath}

In particular, $\rho_{in}, \Delta_\Gamma$ and $\delta^\chi$ are each codimension $\chi d$ embeddings. If $U$ in  $C^{\chi d}(N_\varepsilon, N_{\frac{\varepsilon}{2}})$ represents the Thom class of the multidiagonal $\delta^\chi: M^\chi \to M^{2\chi}$, then $$f_\Gamma^*(U) \in C^{\chi d}(f_\Gamma^{-1}(N_\varepsilon), f_\Gamma^{-1}(N_{\frac{\varepsilon}{2}}))$$ represents the Thom class of $$\rho_{in} :Maps(S(\Gamma), M) \to LM^k.$$ Additionally, 
$$ev_{c_\Gamma}^*(U) = \sigma^*(f_\Gamma^*(U)) \in C^{\chi d}(\tilde{S}_\varepsilon, \tilde{S}_\varepsilon - \tilde{S}_{\frac{\varepsilon}{2}})$$
where $f_\Gamma = i_\Gamma \circ e_\Gamma$.
For the neighborhood $N_\varepsilon$ of the multidiagonal $\delta^\chi(M^\chi) \subset M^{2\chi}$, $f_\Gamma^{-1}(N_\varepsilon)$ is a neighborhood of $\rho_{in}(Maps(S(\Gamma), M)) \subset LM^k$ and $$\tilde{S}_\varepsilon = ev_{c_\Gamma}^{-1}(N_\varepsilon) = \sigma^{-1}(f_\Gamma^{-1}(N_\varepsilon)) \subset \Delta^n.$$

The umkehr map on homology $$(\rho_{in})_!: h_*(LM^k) \to h_{*-\chi d}(Maps(S(\Gamma), M))$$ is defined in \cite{CG} by composing the map induced on homology by a Thom collapse and the Thom isomorphism. Indeed, for $h_* = H_*$, we realize the map $(\rho_{in})_!$ as induced on homology by a specific composition of chain maps. Let $f_\Gamma^{-1}(N_\varepsilon) = F_\varepsilon$ and $f_\Gamma^{-1}(N_\frac{\varepsilon}{2}) = F_\frac{\varepsilon}{2}$ be the tubular neighborhoods of $Im(\rho_{in})$.

Consider the following commutative diagram of spaces.
\begin{displaymath}
    \xymatrix{(LM^k, \emptyset) \ar[rr]^{J} \ar[d]^\tau
    &&  (LM^k, LM^k - F_\frac{\varepsilon}{2}) \ar[d]^q
 \\ 
  (LM^k/(LM^k - F_\frac{\varepsilon}{2}), \emptyset)\ar[rr]^{\hspace{-2cm}J'}
             && (LM^k/(LM^k - F_\frac{\varepsilon}{2}),  (LM^k - F_\frac{\varepsilon}{2})/(LM^k - F_\frac{\varepsilon}{2}))\\
                 }
\end{displaymath}
Here, $J$ and $J'$ are induced by inclusions of the empty set into the appropriate spaces, $\tau$ is the Thom collapse map and $q$ is the quotient map. The maps induced by $J'$ and $q$ on homology are isomorphisms in dimensions $>0$. In particular if $q_\#$ is the induced map on chains, $q_\#$ has a chain homotopy inverse in dimensions $>0$. Call this map $q_\#^{-1}$. Then the following diagram commutes up to chain homotopy.
\begin{displaymath}
    \xymatrix{C_*(LM^k, \emptyset) \ar[rr]^{J_\#} \ar[d]^{\tau_\#}
    &&  C_*(LM^k, LM^k - F_\frac{\varepsilon}{2}) 
 \\ 
  C_*(LM^k/(LM^k - F_\frac{\varepsilon}{2}), \emptyset)\ar[rr]^{\hspace{-2cm}J'_\#}
             && C_*(LM^k/(LM^k - F_\frac{\varepsilon}{2}),  (LM^k - F_\frac{\varepsilon}{2})/(LM^k - F_\frac{\varepsilon}{2})) \ar[u]^{q_\#^{-1}}
             \\
                 }
\end{displaymath}
In particular, the maps induced on homology $J_*$ and $q_*^{-1} \circ J'_* \circ \tau_*$, are equal in dimensions $>0$ and $q_*^{-1}$ and $J'_*$ are isomorphisms in dimensions $>0$. 

The Thom isomorphism is induced by the composition of the following chain maps.
\begin{enumerate}
\item
$$S: C_*(LM^k, LM^k - F_\frac{\varepsilon}{2})  \to C_*(F_\varepsilon, F_\varepsilon - F_\frac{\varepsilon}{2})$$
is given by excising $LM^k - F_\varepsilon$.
\item
$$\cap f^*_\Gamma(U): C_*(F_\varepsilon, F_\varepsilon - F_\frac{\varepsilon}{2}) \to  C_{* - \chi d}(F_\varepsilon)$$
is given by capping with the pulled-back Thom class representative.
\item
$$p_\#: C_{*-\chi d}(F_\varepsilon) \to C_{*-\chi d}(Maps(S(\Gamma), M))$$
is the map induced on by the projection map
 $p: F_\varepsilon \to Maps(S(\Gamma), M)$. Here we are implicitly using the diffeomorphism between $F_\varepsilon$ and the pulled-back normal bundle.
\end{enumerate}
The chain map $p_\# \circ \cap f^*_\Gamma(U) \circ S$ induces the Thom isomorphism on homology and $$p_\# \circ \cap f^*_\Gamma(U) \circ S \circ J_\#$$ induces $$(\rho_{in})_!: H_*(LM^k) \to H_{*-\chi d}(Maps(S(\Gamma), M)).$$
Notice that since $\cap f^*_\Gamma(U)$ has degree $-\chi d<0$, we need only be concerned with dimensions $>0$.

The rest of the proof relies on the homotopy commutativity of a large diagram of chain complexes and chain maps. For simplicity, we indicate only the maps explicitly here. The composition of maps in the top row gives the chain map inducing $(\rho_{in})_!$ as described above. The maps in the bottom row are those in the definition of the string topology construction Section \ref{STC} when $c = \{x_\Gamma\}$ is a 0-cell.
In what follows we prove homotopy commutativity of three sub-diagrams one-by-one. Let $\sigma: \Delta^n \to LM^k$ be a generator of $C_n(LM^k)$.

\begin{displaymath}
    \xymatrix{\bullet\ar[r]^{J_\#} 
    & \bullet \ar[r]^S
    & \bullet \ar[r]^{\cap f_\Gamma^*(U)} 
     & \bullet \ar[rr]^{p_\#}
     && \bullet \ar[drr]^{(\rho_{out})_\#} \ar[dd]_{inc_\#}
 \\ 
     & \hspace{1.2cm}(1) &&&(2)&& (3) & \bullet \\
\bullet \ar[r]^{ j_\#} \ar[uu]^{\sigma_\#}
             &    \bullet   \ar[r]^{s}
              & \bullet \ar[r]^{\cap ev_{c_\Gamma}^*(U)}
                            &\bullet  \ar[rr]^{\alpha^{in}_\#} \ar[uu]^{\sigma_\#}
                            &&\bullet  \ar[urr]_{out_\#}  \\
                 }
\end{displaymath}

\textbf{Diagram (1):}

Diagram (1) is in fact a sequence of three commutative squares. Vertical maps are all induced by $\sigma: \Delta^n \to LM^k$. Indeed, $\sigma$ sends 
\begin{itemize}
\item
$(\Delta^n,  \tilde{S}_{\frac{\varepsilon}{2}}) \to (LM^k, LM^k - F_{\frac{\varepsilon}{2}})$
\item
$(\tilde{S}_\varepsilon, \tilde{S}_{\frac{\varepsilon}{2}}) \to (F_\varepsilon, F_\varepsilon-F_{\frac{\varepsilon}{2}})$
\item
$\tilde{S}_\varepsilon \to F_\varepsilon$
\end{itemize}
By abuse of notation, we call all induced chain maps $\sigma_\#$.
\begin{displaymath}
    \xymatrix{C_*(LM^k) \ar[r]^{\hspace{-1cm}J_\#}
    & C_*(LM^k, LM^k - F_{\frac{\varepsilon}{2}}) \ar[r]^{\; \; S}
     & C_*(F_\varepsilon, F_\varepsilon-F_{\frac{\varepsilon}{2}}) \ar[r]^{\; \; \; \cap f_\Gamma^*(U)} 
     & C_{*- \chi d}(F_\varepsilon)
             \\ 
     &&\\
               C_*(\Delta^n) \ar[r]^{\hspace{-1cm} j_\#} \ar[uu]^{\sigma_\#}
               & C_*(\Delta^n, \Delta^n - \tilde{S}_{\frac{\varepsilon}{2}}) \ar[r]^s \ar[uu]^{\sigma_\#}
             &    C_*(\tilde{S}_\varepsilon, \tilde{S}_{\frac{\varepsilon}{2}})    \ar[r]^{\cap ev_{c_\Gamma}^*(U)} \ar[uu]^{\sigma_\#}
              & C_{*-\chi d}(\tilde{S}_\varepsilon) \ar[uu] \ar[uu]^{\sigma_\#}
              &    \\
                 }
\end{displaymath}

The first two squares of Diagram (1) clearly commute. The third square commutes because $ev_{c_\Gamma}^*(U) = \sigma^*(f_\Gamma^*(U))$.

\textbf{Diagram (2):}

\begin{displaymath}
    \xymatrix{C_{*- \chi d}(F_\varepsilon) \ar[rr]^{p_\#} 
    &&  C_{*- \chi d}(Maps(S(\Gamma), M)) \ar[dd]^{(inc_\Gamma)_\#}
 \\  & \\
C_{*- \chi d}(\tilde{S}_\varepsilon)\ar[uu]^{\sigma_\#}  \ar[rr]^{\alpha^{in}_\#}
             && C_{* - \chi d}(Maps(S_\varepsilon, M))   \\
                 }
\end{displaymath}

Since $c = c_\Gamma = x_\Gamma$ is a 0-cell, $c \times \Delta^n = \{x_\Gamma\} \times \Delta^n \cong \Delta^n$, $\pi: c \times \Delta^n \to c$ is the map $\pi: \{x_\Gamma\} \times \Delta^n \to \{x_\Gamma\}$. If  $\tilde{S}_\varepsilon = \emptyset$ then the operation is identically zero. If $\tilde{S}_\varepsilon \neq \emptyset$ then $S_\varepsilon = \{x_\Gamma\}$ and $\mathcal{M}aps(S_\varepsilon, M) = Maps(\Gamma, M)$ is the usual mapping space. The map 
$$inc_\Gamma: Maps(S(\Gamma), M) \to Maps(\Gamma, M)$$ is then just the inclusion of maps that are constant on chords of $\Gamma$ and $(inc_\Gamma)_\#$ is the induced map on chains.

Diagram (2) is induced by a diagram of spaces and continuous maps. This diagram commutes up to homotopy.
\begin{displaymath}
    \xymatrix{F_\varepsilon \ar[rr]^{\hspace{-1cm}p} 
    &&  Maps(S(\Gamma), M) \ar[dd]^{inc_\Gamma}
 \\  & \\
\tilde{S}_\varepsilon\ar[uu]^{\sigma}  \ar[rr]^{\alpha^{in}}
             && Maps(\Gamma, M) \\
                 }
\end{displaymath}
Let $t \in \tilde{S}_\varepsilon \subset \Delta^n$. Then $\sigma(t) \in f_\Gamma^{-1}(N_\varepsilon) \subset LM^k$ determines a map $ \bigsqcup_k S^1 \to M$ which, by abuse of notation we also call $\sigma(t)$, such that for all chords $e_\Gamma$ of $\Gamma$, $\sigma(t)(\varphi_{e_\Gamma}(0))$ and $\sigma(t)(\varphi_{e_\Gamma}(1))$ lie in some $\varepsilon$ ball in $M$. The map $\alpha^{in}(t)$ maps $\Gamma$ to $M$ by mapping its input circles via $\sigma(t)$ and mapping each chord $e$ to the unique geodesic segment joining $\sigma(t)(\varphi_e(0))$ and $\sigma(t)(\varphi_e(1))$  and lying in the $\varepsilon$ ball as described in Section \ref{STC}.

The projection map $p: f_\Gamma^{-1}(N_\varepsilon) \to Maps(S(\Gamma), M)$ is a deformation retraction. It
takes the map 
$\sigma(t): \bigsqcup_k S^1 \to M$,  such that for all chords ${e_\Gamma}$ of $\Gamma$, $\sigma(t)(\varphi_{e_\Gamma}(0))$ and $\sigma(t)(\varphi_{e_\Gamma}(1))$ lie in some $\varepsilon$ ball, to a map $p(\sigma(t)): \bigsqcup_k S^1 \to M$, such that for all ${e_\Gamma}$, $\sigma(t)(\varphi_{e_\Gamma}(0))=\sigma(t)(\varphi_{e_\Gamma}(1))$. In particular, $\sigma(t)$ and $p(\sigma(t))$ are homotopic.

This homotopy extends to a homotopy $H_t: \Gamma \times I \to M$ between the maps of $\Gamma$ to $M$: $\alpha^{in}(t)$ and  $inc_\Gamma( p ( \sigma(t)))$. Then $$H: \tilde{S}_\varepsilon \times I \to Maps(\Gamma, M), \;(t, s) \mapsto H_t(s)$$ is a homotopy between $\alpha^{in}$ and $inc_\Gamma \circ p \circ \sigma$. This shows that Diagram (2) commutes up to the chain homotopy induced by $H$. By abuse of notation, denote the chain homotopy by $H$ as well, so
$$\partial H + H \partial = (inc_\Gamma)_\# \circ p_\# \circ \sigma_\#.$$

\textbf{Diagram (3)}

\begin{displaymath}
    \xymatrix{      C_{* - \chi d}(Maps(S(\Gamma), M)) \ar[drr]^{(\rho_{out})_\#} \ar[dd]_{(inc_\Gamma)_\#}
 \\ 
     &&C_{*-\chi d}(LM^\ell))\\
                                         C_{*- \chi d}(Maps(\Gamma, M))  \ar[urr]_{out_\#}  \\
                 }
\end{displaymath}

Diagram (3) is induced by a strictly commutative diagram of spaces and continuous maps:

\begin{displaymath}
    \xymatrix{      Maps(S(\Gamma), M) \ar[drr]^{\rho_{out}} \ar[dd]_{inc_\Gamma}
 \\ 
     &&(LM^\ell)\\
                                         Maps(\Gamma, M)  \ar[urr]_{out}  \\
                 }
\end{displaymath}

Let $$\tilde{\mu_\Gamma} = (\rho_{out})_\# \circ p_\# \circ \cap f^*_\Gamma(U) \circ S \circ J_\#$$
and $$\lambda_{\Gamma, \sigma} = out_\# \circ \alpha^{in}_\# \circ \cap ev_{c_\Gamma}^*(U) \circ s \circ j_\#.$$

The (homotopy) commutativity of Diagrams  (1), (2), and (3) tells us that
$$K: C_*(\Delta^n) \to C_{* - \chi d + 1}(LM^\ell), \; \;K = out_\# \circ H \circ \cap ev_{c_\Gamma}^*(U) \circ s \circ j_\#$$ 
satisfies $$\partial K - (-1)^{- \chi d + 1} K \partial = \tilde{\mu_\Gamma} \circ \sigma_\# - \lambda_{\Gamma, \sigma}.$$
Thus, $K$ is a chain homotopy between
$$\tilde{\mu_\Gamma} \circ \sigma_\# \textrm{ 
and }\lambda_{\Gamma, \sigma} .$$

Recall that 
$$\lambda_\Gamma : C_*(LM^k) \to C_{*- \chi d}(LM^\ell), \; \sum_i a_i \sigma_i \mapsto \sum_i a_i \lambda_{\Gamma, \sigma_i}(\mu_{\Delta^n})$$
where $\mu_{\Delta^n}$ is the fundamental chain of $\Delta^n$.

We use $K$ to build a  chain homotopy 
$$K':C_*(LM^k) \to C_{* - \chi d + 1}(LM^\ell)$$ 
between $\tilde{\mu}_\Gamma $ and $\lambda_\Gamma$. 

For a generator $\sigma: \Delta^n \to LM^k$ of $C_*(LM^k)$ let $K'(\sigma) = (-1)^{-\chi d+1} K(\mu_{\Delta^n})$. Then 
\begin{align*}
K'(\partial \sigma) &= (-1)^{-\chi d+1} K(\partial \mu_{\Delta^n}) \\
&= \partial K\mu_{\Delta^n} - \tilde{\mu}_\Gamma ( \sigma_\#(\mu_{\Delta^n})) + \lambda_{\Gamma, \sigma}(\mu_{\Delta^n})\\
& = (-1)^{-\chi d+1}  \partial K'(\sigma) - \tilde{\mu}_\Gamma(\sigma) + \lambda_\Gamma(\sigma)\\ \\
\textrm{so } K' \partial &= (-1)^{-\chi d+1} \partial K' - \tilde{\mu}_\Gamma + \lambda_\Gamma
\end{align*}

Therefore, $K'$ is a chain homotopy between $\lambda_\Gamma$ and $\tilde{\mu}_\Gamma$ and both maps induce $$\mu_\Gamma: H_*(LM)^{\otimes k} \cong H_*(LM^k) \to H_{* - \chi d}(LM^\ell) \cong H_{* - \chi d}(LM)^{\otimes \ell}.$$
\end{proof}

Recall the definition of the chain-level string bracket $\mathcal{ST}(c_\Gamma, -)$ of Example \ref{loopproduct}.
Since $c_\Gamma$ is a cycle, the chain-level loop product is a chain map and so it induces a product
$$H_*(LM)^{\otimes 2} \to H_{*-d}(LM).$$

\begin{cor}\label{corcor}
The chain-level loop product induces the Chas-Sullivan loop product
$$\bullet: H_i(LM) \otimes H_j(LM) \to H_{i+j-d}(LM).$$ We obtain an isomorphism $$H_*(LM) \to H_*(LM)$$ of commutative algebra structures.
\end{cor}

\begin{remark}
Together with the BV operator on $H_*(LM)$ induced by the $S^1$ action on $LM$, we recover Chas and Sullivan's original BV algebra structure on $H_*(LM)$.
\end{remark}

The operation $\mu_\Gamma$ of \cite{CG} depends only on $\gkl$. 
The proof of this fact in \cite{CG} uses the fact that $Sull$ is path connected. The construction may be generalized to $$\mu_\alpha: h_*(LM^k)  \to h_{* + |\alpha| - \chi d}(LM^\ell)$$ for any $\alpha \in h_*(Sull)$. Details of a generalization do not appear in \cite{CG} but do in \cite{ChataurBordism}.
The operation $\mu_\Gamma$ is equal to $\mu_1$ for $1$ a generator of $h_0(Sull)$.

The construction of Section \ref{STC} induces operations $$H_*(LM)^{\otimes k} \to H_{* + |\alpha| - \chi d}(LM)^{\otimes \ell}$$ coming from classes $\alpha$ in $H_*(\sdbar)$. We are interested in comparing these operations on homology to those previously defined \cite{CG, ChataurBordism}. We examine the operations coming from $0$-dimensional homology classes in detail. The space $\sdbar$ is disconnected in general; we will define an equivalence relation $\sim$ on $\sdbar$ such that the quotient $\sdbar / \sim$ is connected. This quotient is a compactification of a space homotopy equivalent to Cohen-Godin's space $Sull$.

\begin{definition}
Let $\Gamma$ and $\Gamma'$ be two string diagrams. $\Gamma$ and $\Gamma'$ differ by a slide if they are identical except for the attaching map of one chord: $e$ in $\Gamma$ and $e'$ in $\Gamma'$ are related as follows. Assume in $\Gamma$, $\varphi_e(a) = \varphi_f(b)$ for $a,b \in \{0,1\}$, where the chord $f$ follows (respectively precedes) $e$ in the cyclic order at this vertex. Then in $\Gamma'$, $\varphi_{e'}(a) = \varphi_f(c)$ for $c \in \{0,1\}, c \neq b$ and $e'$ precedes (respectively follows) $f$ in the cyclic order at this vertex.
\end{definition}

 \begin{figure}[h]\label{slideslidea}
\centering
\includegraphics{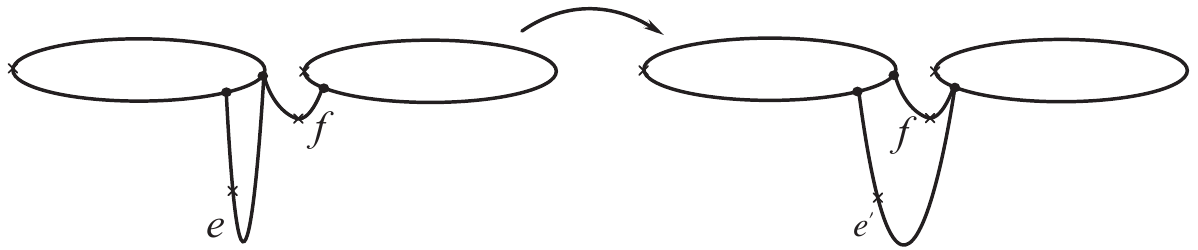}
\caption{Two string diagrams that differ by a slide.}
\end{figure}

\begin{definition}
Let slides generate an equivalence relation $\sim$ on $\sdbar$.
\end{definition}

\begin{remark}
Slide equivalence induces an equivalence relation on the set of cells of $\sdbar$ and hence on the set of generators of the cellular chains $\mathcal{C}_*(\sdbar)$.
\end{remark}

\begin{definition}
Let $\sdbarr$ denote $\sdbar/\sim$ and let $\mathscr{SD}$ denote $SD/\sim$.
\end{definition}

\begin{definition}
If $z$ is a chain in  $\mathcal{C}_*(\sdbar)$, let 
\[\mathcal{ST}(z,- ): C_*(LM^k) \to C_{* + m - \chi d}(LM^\ell)\]
be defined by $\mathcal{ST}(z,-)(\s) = \mathcal{ST}(z,\s)$.
\end{definition}
\begin{remark}
If $z$ is an $m$-cycle in $\mathcal{C}_*(\sdbar)$, then $\mathcal{ST}(z,-): C_*(LM^k) \to C_{* + m - \chi d}(LM^\ell)$ is a degree $m-\chi d$ chain map. 
\end{remark}

We wish to compare $\mathcal{ST}(z,-)$ and $\mathcal{ST}(z',-)$ when $z$ and $z'$ are slide equivalent cycles. Below, we define a map 
$$\Lambda_{z, z'}: C_*(LM^k) \to C_{* +m - \chi d + 1}(LM^\ell)$$ and in Proposition \ref{slideprop} we show it is a chain homotopy between $\mathcal{ST}(z,-)$ and $\mathcal{ST}(z',-)$.

Fix a generator $\sigma: \Delta^n \to LM^k$ of $C_n(LM^k)$ and cells $c$ and $c'$ of $\sdbar$ that differ by a slide. Recall that the construction of $\mathcal{ST}$ uses a map $g_{(c, \sigma)}$, which is a composition of chain maps $C_*(c \times \Delta^n) \to C_{*-\chi d}(LM^k)$. Analogously, the construction of the map $\Lambda_{c,c'}(\sigma)$ below uses a map $L_{c, c', \sigma}$, which is a composition of chain maps $C_*(I) \otimes C_*(c \times \Delta^n)  \to C_{*-\chi d}(LM^k)$. Most of the maps in the composition are completely analogous to those in the definition of $g_{(c,\sigma)}$. The cells $c$ and $c'$ are fixed for this discussion and so we drop them from the notation and let $L_\sigma$ be the following composition of chain maps.

The first map in the composition defining $L_\sigma$ is the Eilenberg-Zilber map $$EZ: C_*(I) \otimes C_*(c \times \Delta^n) \to C_*(I \times c \times \Delta^n).$$

Here we describe $c \times I$ as a space of marked metric fatgraphs such that $c \times \{0\} \sim c$ and $c \times \{1\} \sim c'$. Consider $(x_\Gamma, t) \in {x_\Gamma} \times I$. Assume the $i$-th chord $e_i$ of $\Gamma$ slides over the chord $e$ to produce $\Gamma'$. Assume $\varphi_{e_i}(0) = \varphi_{e}(0)$ in $\Gamma$,  $\varphi_{e_i}(0) = \varphi_{e}(1)$ in $\Gamma'$ and that $\varphi_{e_i}(1)$ in $\Gamma$ and $\Gamma'$ are equal. Under the identification of $e$ with $[0,1]$ we abuse notation and write $\varphi_{e}(0) = 0$ and $\varphi_{e}(0) = 1$.

Let $\Gamma_s$ be the marked metric fatgraph produced when $\varphi_{e_i}(0) = s$. Notice $\Gamma_0 = \Gamma$ and $\Gamma_1 = \Gamma'$. Let $x_{\Gamma_s} = (s, x_\Gamma) \in I \times c$ refer to the graph $\Gamma_s$.

 \begin{figure}[h]\label{slidea}
\centering
\includegraphics{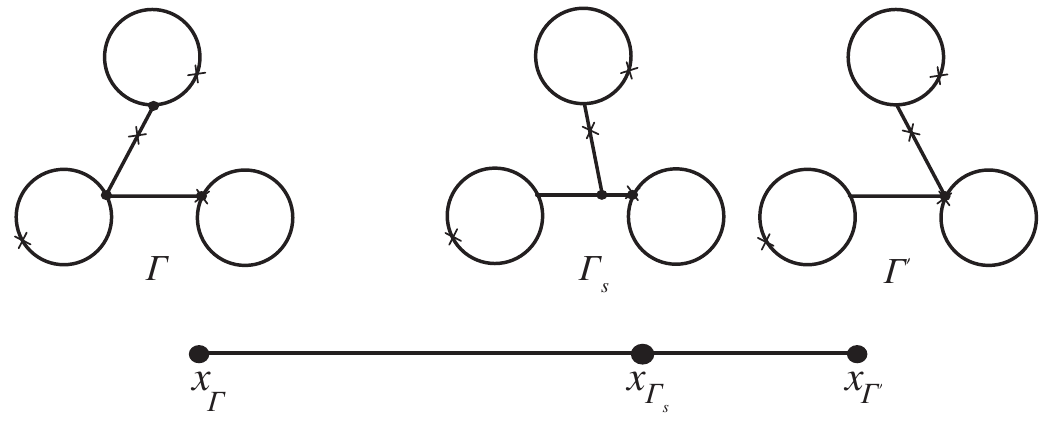}
\caption{$\{x_\Gamma\} \times I$ and the corresponding fatgraphs.}
\end{figure}

To describe the next map in the composition defining $L_\sigma$, we first need to define an evaluation map. This will be done in two steps.

For any $x_{\Gamma_s} \in  I \times c$, the deletion of the $i$-th chord $e_i$ from $\Gamma_s$ yields a string diagram $\Gamma^i$ of type $(g_i, k_i, \ell_i)$. 
Let $x_{\Gamma^i}$ be the corresponding point in the space $\sdbar(g_i, k_i, \ell_i)$ of string diagrams of type $(g_i, k_i, \ell_i)$. Let 
\begin{align*}
F^i: \sdbar &\to \sdbar (g_i, k_i, \ell_i)\\
x_\Gamma &\mapsto x_{\Gamma^i}.
\end{align*}

We define a preliminary evaluation map 
\begin{align*}
\widetilde{ev}^n(\sigma): I \times c \times \Delta^n &\to M^{2\chi - 2}\\
(s, x_\Gamma, t) &\mapsto ev^n(\sigma)(x_{\Gamma^i}, t) 
\end{align*}
where $$ev^n: Maps(\sqcup_k \times \Delta^n, M) \to Maps (\sdbar(g_i, k_i, \ell_i) \times \Delta^n, M^{2\chi -1})$$
is as in Section \ref{STC}.

Let 
$\tilde{T}_\varepsilon = \{(s, x_\Gamma, t) | \widetilde{ev}^n(\sigma)(s, x_\Gamma, t) \in N_\varepsilon\}$, where $N_\varepsilon \subset M^{2\chi - 2}$ is the $\varepsilon$ neighborhood of the multidiagonal $\delta^{\chi - 1}(M^{\chi - 1}) \subset M^{2 \chi -2}$. (We will abuse notation and use $N_\varepsilon$ again in a moment for the $\varepsilon$ neighborhood of the multidiagonal \newline $\delta^{\chi - 1}(M^{\chi}) \subset M^{2\chi}$.)

Let $\pi^i: I \times c \times \Delta^n \to c$ be the projection map and let $T_\varepsilon = \pi^1(\tilde{T}_\varepsilon)$. 

The map $(\alpha^{in})^i: \tilde{T}_\varepsilon \to Maps(F^i(\pi^i(\tilde{T}_\varepsilon)), M)$ is analogous to the map $\alpha^{in}$ of section \ref{STC}: for each point of $x_{\Gamma_s} \in \tilde{T}_\varepsilon$ we are mapping the string diagram $\Gamma^i$ to $M$. In particular, for $(s, x_\Gamma, t)$, the chord $e$ of $\Gamma^i$ is mapped to a short geodesic segment $\gamma$ in $M$ joining $\sigma(\varphi_e(0))$ and $\sigma(\varphi_e(1))$.

We are now prepared to define the evaluation map
$ev^n_{c, c'}(\sigma): \tilde{T}_\varepsilon \to M^{2 \chi}.$
Let 
$ev^n_{c, c'}(\sigma)(s, x_\Gamma, t)_j$ denote the $j$th coordinate of  $ev^n_{c, c'}(\sigma)(s, x_\Gamma, t)$ and $\widetilde{ev}^n(\sigma)_j$ denote the $j$th coordinate of $\widetilde{ev}^n(\sigma)$.
Then
\begin{displaymath}
  ev^n_{c, c'}(\sigma)(s, x_\Gamma, t)_j =  \left\{
     \begin{array}{lr}
       \widetilde{ev}^n(\sigma)_j & j < 2i - 1\\
       \widetilde{ev}^n(\sigma)_{j-2} & j > 2i\\
      \sigma(\varphi_{e_i}(1))& j =2i\\
       \gamma(s)& j  = 2i-1
     \end{array}
   \right.
\end{displaymath}

Let $$\tilde{U}_\varepsilon = \{(s, x_\Gamma, t) : ev_{c,c'}^n(\sigma)(s, x_\Gamma, t) \in N_\varepsilon \subset M^{2 \chi}\}$$
and $$ev_{c,c'} = ev_{c,c'}^n(\sigma)|_{(\tilde{U}_\varepsilon, \tilde{U}_\varepsilon - \tilde{U}_\frac{\varepsilon}{2})}: (\tilde{U}_\varepsilon, \tilde{U}_\varepsilon - \tilde{U}_\frac{\varepsilon}{2}) \to (N_\varepsilon, N_\varepsilon - N_\frac{\varepsilon}{2}).$$

The next three maps in the composition defining $L_\sigma$ are completely analogous to those defined in section \ref{STC} and we use similar notation.
\begin{enumerate}
\item
$j^I_\#: C_*(I \times c \times \Delta^n) \to C_*(I \times c \times \Delta^n,I \times c \times \Delta^n - \tilde{U}_\frac{\varepsilon}{2} )$ is the quotient map induced by inclusion $(I \times c \times \Delta^n, \emptyset) \to (I \times c \times \Delta^n,I \times c \times \Delta^n - \tilde{U}_\frac{\varepsilon}{2} )$.
\item
$s^I: C_*(I \times c \times \Delta^n,I \times c \times \Delta^n - \tilde{U}_\frac{\varepsilon}{2} ) \to C_*(\tilde{U}_\varepsilon,\tilde{U}_\varepsilon - \tilde{U}_\frac{\varepsilon}{2} )$
is given by excision.
\item
$\cap ev_{c, c'}^*(U): C_*(\tilde{U}_\varepsilon,\tilde{U}_\varepsilon - \tilde{U}_\frac{\varepsilon}{2} ) \to C_{*- \chi d}(\tilde{U}_\varepsilon )$
is the cap product with the pulled-back Thom class representative $U$.
\end{enumerate}

We modify the definition of the map $\alpha^{in}$ only slightly in this setting, again using similar notation.

Let $p: I \times c \times \Delta \to I \times c$ be the projection and let $U_\varepsilon = p(\tilde{U}_\varepsilon))$. 
We define 
$$(\alpha^{in})^I: \tilde{U}_\varepsilon \to Maps(U_\varepsilon, M)$$
$$(s, x_\Gamma, t) \mapsto f_{(s, x_\Gamma, t)}: \Gamma_s \to M.$$
Again, the map $f_{(s, x_\Gamma, t)}$ is given by pasting together maps on input circles and chords.
Input circles of $\Gamma_s$ are mapped via $\sigma: \sqcup_k S^1 \times \Delta^n \to M$:
$$f_{(s, x_\Gamma, t)}|_{\sqcup_k S^1} = \sigma_t.$$
On the $j$-th chord, $j \neq i$, we follow the geodesic $\gamma^j$ joining $\sigma_t(\varphi_{e_j}(0))$ and $\sigma_t(\varphi_{e_j}(1))$.
On the $i$-th chord, we follow the geodesic $\gamma^i$ joining $\sigma_t(\varphi_{e_i}(1))$ and $\gamma(s)$. 

We are ready to define the last two maps in the composition giving $L_\sigma$.
\begin{enumerate}
\item
$(\alpha^{in})^I_\#: C_{*- \chi d}(\tilde{U}_\varepsilon ) \to C_{*- \chi d}(Maps(U_\varepsilon, M ))$ is the map induced on chains by $(\alpha^{in})^I$.
\item
$out^I_\#: C_{*- \chi d}(Maps(U_\varepsilon, M )) \to C_{*- \chi d}(LM^\ell )$ is the map induced by the restriction $out^I: Maps(U_\varepsilon, M ) \to LM^\ell$ to outputs as usual.
\end{enumerate}

\begin{definition}
Let $$L_\sigma = out^I_\# \circ ((\alpha^{in})^I)_\# \circ \cap ev_{c, c'}^*(U) \circ s^I \circ j^I_\# \circ EZ: C_*(I) \otimes C_*(c \times \Delta) \to C_{*- \chi d}(LM^\ell ).$$
\end{definition}

\begin{remark}
$L_\sigma$ is a degree $-\chi d$ chain map, that is, $$\partial \circ L_\sigma = (-1)^{- \chi d}L_\sigma \circ \partial.$$
\end{remark}

\begin{lemma}\label{Lsigma}
Let $\partial_i: \Delta^{n-1} \to \Delta^n$ be the inclusion of the $i$-th face given by omitting the $i$-th vertex and let $\sigma_i$ be the restriction of $\sigma_i = \sigma \circ \partial_i$. Let $\partial_{r,s}: \partial_{rs}c \to c$ correspond to the $s$th face map of the $r$th simplex factor of $c$ as in section \ref{STC}. Then the map $L_{\sigma,c,c'}$ satisfies
\begin{enumerate}
\item $L_{\sigma,c,c'}((\partial_0)_\#(\mu_{\Delta^0}) \otimes \mu_{c \times \Delta^n}) = g_{(c, \sigma)}(\mu_{c \times \Delta^n}) = \mathcal{ST}(c, \sigma)$
\item $L_{\sigma,c,c'}((\partial_1)_\#(\mu_{\Delta^0}) \otimes \mu_{c \times \Delta^n}) = g_{(c', \sigma)}(\mu_{c' \times \Delta^n}) = \mathcal{ST}(c', \sigma)$

\item 
$L_{\sigma, c, c'}(\mu_I \otimes EZ((\partial_{rs})_\#(\mu_{\partial_{rs} c}) \otimes \mu_{\Delta^{n}}))
\\ = L_{\sigma, \partial_{rs}c, \partial_{rs}c'}(\mu_I \otimes EZ(\mu_{\partial_{rs}c} \otimes \mu_{\Delta^{n}}))$

\item $L_{\sigma,c,c'}(\mu_I \otimes EZ(\mu_c \otimes (\partial_i)_\#(\mu_{\Delta^{n-1}}))) = L_{\sigma_i, c,c'}(\mu_I \otimes \mu_{c \times \Delta^{n-1}})$.
\end{enumerate}
\end{lemma}

\begin{proof}

The proof of the four statements relies on the fact that 
$$\partial(I \times c \times \Delta^n)= (\{1\} \times c \times \Delta^n)\cup (\{0\} \times c \times \Delta^n) \cup (I \times \partial c \times \Delta^n) \cup (I \times c \times \partial \Delta^n).$$
Recall that we have described $I \times c$ as a space of marked metric fatrgraphs such that $\{0\} \times c \sim c$ and $\{1\} \times c \sim c'$. We have the maps
$$ \partial_0 \times id: c' \times \Delta^n \sim \{1\} \times c' \times \Delta^n  \to I \times c \times \Delta^n$$
and 
$$ \partial_1 \times id: c \times \Delta^n \sim \{0\} \times c \times \Delta^n  \to I \times c \times \Delta^n$$

The evaluation map 
$ev^n_{c,c'}(\sigma): \tilde{T}_\varepsilon \to M^{2 \chi}$
satisfies
\begin{enumerate}
\item
$ev^n_{c,c'}(\sigma)\circ (\partial_0 \times id)|_{(\partial_0 \times id)^{-1}(\tilde{T}_\varepsilon)} = ev_{c'}^n(\sigma)|_{(\partial_0 \times id)^{-1}(\tilde{T}_\varepsilon)}$

\item
$ev^n_{c,c'}(\sigma)\circ (\partial_1 \times id)|_{(\partial_1 \times id)^{-1}(\tilde{T}_\varepsilon)} = ev_{c}^n(\sigma)|_{(\partial_1 \times id)^{-1}(\tilde{T}_\varepsilon)}$

\item
$ev^n_{c,c'}(\sigma) \circ (id \times \partial_{rs} \times id)|_{ (id \times \partial_{rs} \times id)^{-1}(\tilde{T}_\varepsilon)} = ev^n_{\partial_{rs}c, \partial_{rs}c'}(\sigma)$

\item
$ev^n_{c,c'}(\sigma) \circ (id \times id \times \partial_i)|_{(id \times id \times \partial_i)^{-1}(\tilde{T}_\varepsilon)} = ev^{n-1}_{c,c'}(\sigma_i)|_{(id \times id \times \partial_i)^{-1}(\tilde{T}_\varepsilon)}$
\end{enumerate}

Therefore,
\begin{enumerate}
\item
$\partial_0 \times id:  c' \times \Delta^n \sim \{1\} \times c \times \Delta^n \to I \times c \times \Delta^n$
satisfies
$$(\tilde{S'}_\varepsilon, \tilde{S'}_\varepsilon - \tilde{S'}_\frac{\varepsilon}{2}) \to  (\tilde{U}_\varepsilon, \tilde{U}_\varepsilon - \tilde{U}_\frac{\varepsilon}{2})$$

\item
$\partial_1 \times id: c \times \Delta^n \sim \{0\} \times c \times \Delta^n \to I \times c \times \Delta^n$
satisfies
$$(\tilde{S}_\varepsilon, \tilde{S}_\varepsilon - \tilde{S}_\frac{\varepsilon}{2}) \to  (\tilde{U}_\varepsilon, \tilde{U}_\varepsilon - \tilde{U}_\frac{\varepsilon}{2})$$

\item
$id \times \partial_{rs} \times id: I \times \partial_{rs}c \times \Delta^n \to I \times c \times \Delta^n$
satisfies
$$(\tilde{U}_\varepsilon(\partial_{rs}c, \partial_{rs}c',\sigma),\tilde{U}_\varepsilon(\partial_{rs}c, \partial_{rs}c',\sigma) - \tilde{U}_\frac{\varepsilon}{2}(\partial_{rs}c,\partial_{rs}c',\sigma))  \to (\tilde{U}_\varepsilon(c,c'\sigma),\tilde{U}_\varepsilon(c,c',\sigma) - \tilde{U}_\frac{\varepsilon}{2}(c,c',\sigma))  $$

\item
$id \times id \times \partial_i: I \times c \times \Delta^{n-1} \to I \times c \times \Delta^n$
satisfies
$$(\tilde{U}_\varepsilon(c,c',\partial_i\sigma),\tilde{U}_\varepsilon(c,c',\partial_i\sigma) - \tilde{U}_\frac{\varepsilon}{2}(c,c',\partial_i\sigma))  \to (\tilde{U}_\varepsilon(c,c'\sigma),\tilde{U}_\varepsilon(c,c',\sigma) - \tilde{U}_\frac{\varepsilon}{2}(c,c',\sigma))  $$
\end{enumerate}

These maps induce the vertical maps in the following commutative diagrams. (Again, we suppress names of chain complexes and chain maps when they are clear.)
\begin{displaymath}
    \xymatrix{
    &\bullet\ar[r]^{j_\#}  \ar[ldd]_{(\partial_\iota)_\#(\mu_{\Delta^0}) \otimes} \ar[dd]^{(\partial_\iota \times id)_\#}
    & \bullet \ar[r]^s \ar[dd]^{(\partial_\iota \times id)_\#}
    & \bullet \ar[rr]^{\cap ev_{c}^*(U)} \ar[dd]^{(\partial_\iota \times id)_\#}
     && \bullet \ar[rr]^{\alpha^{in}_\#} \ar[dd]^{(\partial_\iota \times id)_\#}
     && \bullet \ar[drr]^{(out)_\#} \ar[dd]^{(\partial_\iota \times id)_\#}
 \\ 
     && \hspace{1.2cm} &&&&&& & \bullet \\
     \bullet  \ar[r]^{EZ}
&\bullet \ar[r]^{ j^I_\#}
             &    \bullet   \ar[r]^{s^I}
              & \bullet \ar[rr]^{\cap ev_{c,c'}^*(U)}
                            &&\bullet  \ar[rr]^{(\alpha^{in})^I_\#}                        &&\bullet  \ar[urr]_{out^I_\#}  \\
                 }
\end{displaymath}
for $\iota \in \{0,1\}$.

\begin{displaymath}
    \xymatrix{
    C_*(I) \otimes C_*(\partial_{rs}c \times \Delta^n)\ar[r]^{}  \ar[dd]_{id \otimes (\partial_{rs} \times id)_\#} 
    & \bullet \ar[r] \ar[dd]_{(id \times \partial_{rs} \times id)_\#}
    & \bullet \ar[r] \ar[dd]^{}
    & \bullet \ar[rr]^{} \ar[dd]^{}
     && \bullet \ar[r]^{} \ar[dd]^{}
     & \bullet \ar[drr]^{} \ar[dd]^{}
 \\ 
     && \hspace{1.2cm} &&&&& & \bullet \\
     C_*(I) \otimes C_*(c \times \Delta^n)  \ar[r]^{\hspace{1cm} EZ}
&\bullet \ar[r]^{ j^I_\#}
             &    \bullet   \ar[r]^{s^I}
              & \bullet \ar[rr]^{\cap ev_{c,c'}^*(U)}
                            &&\bullet  \ar[r]^{(\alpha^{in})^I_\#}                        &\bullet  \ar[urr]_{out^I_\#}  \\
                 }
\end{displaymath}

\begin{displaymath}
    \xymatrix{
    C_*(I) \otimes C_*(c \times \Delta^{n-1})\ar[r]^{}  \ar[dd]_{id \otimes (id \times \partial_i)_\#} 
    & \bullet \ar[r] \ar[dd]_{(id \times id \times \partial_i)_\#}
    & \bullet \ar[r] \ar[dd]^{}
    & \bullet \ar[rr]^{} \ar[dd]^{}
     && \bullet \ar[r]^{} \ar[dd]^{}
     & \bullet \ar[drr]^{} \ar[dd]^{}
 \\ 
     && \hspace{1.2cm} &&&&& & \bullet \\
     C_*(I) \otimes C_*(c \times \Delta^n)  \ar[r]^{\hspace{1cm} EZ}
&\bullet \ar[r]^{ j^I_\#}
             &    \bullet   \ar[r]^{s^I}
              & \bullet \ar[rr]^{\cap ev_{c,c'}^*(U)}
                            &&\bullet  \ar[r]^{(\alpha^{in})^I_\#}                        &\bullet  \ar[urr]_{out^I_\#}  \\
                 }
\end{displaymath}

By evaluating each diagram on the appropriate chains, we obtain the four statements of the lemma.

\end{proof}

\begin{definition}
\begin{enumerate}
\item
Fix a generator $\sigma: \Delta^n \to LM^k$ of $C_n(LM^k)$ and $m$-cells $c$ and $c'$ of $\sdbar$ that differ by a slide. Let $$\Lambda_{c,c'}(\sigma) = (-1)^{-\chi d}L_\sigma(\mu_I \otimes \mu_{c \times \Delta^n}).$$ Extend $\Lambda_{c,c'}$ to $C_*(LM^k)$ linearly.
\item
Assume $z$ and $z'$ are chains in $C_*(LM^k)$ such that $$z = \sum_j a_j c_j, z' = \sum_j a_j c_j'$$ where $c_j$ and $c_j'$ differ by a slide for all $j$. Define $$\Lambda_{z,z'}: C_*(LM^k) \to C_{* - \chi d +1}(LM^\ell)$$ by
$\Lambda_{z,z'}(\alpha) = \sum_j a_j \Lambda_{c_j, c_j'}(\alpha)$ for $\alpha \in C_*(LM^k)$.
\end{enumerate}
\end{definition}

\begin{prop}\label{slideprop}
If $z$ and $z'$ are slide-equivalent cycles in $\mathcal{C}_*(\sdbar)$, then 
$\Lambda_{z, z'}$ is a chain homotopy between $\mathcal{ST}(z,-)$ and $\mathcal{ST}(z',-)$.
\end{prop}

\begin{proof}

We need to show that $$\partial \Lambda_{z,z'} - (-1)^{- \chi d +m+1} \Lambda_{z,z'} \partial = \mathcal{ST}(z,-) - \mathcal{ST}(z',-).$$

First, fix $m$-cells $c$ and $c'$ of $\sdbar$ that differ by a slide.
\begin{align*}
\Lambda_{c,c'} (\partial \sigma) =& \Lambda_{c,c'}\left(\sum_i (-1)^i \sigma_i\right)\\
=& \sum(-1)^i \Lambda_{c,c'}(\sigma_i)\\
=& (-1)^{\chi d}\sum_i (-1)^i L_{\sigma_i} (\mu_I \otimes \mu_{c \times \Delta^{n-1}})
\end{align*}
\begin{align*}
\partial \Lambda_{c,c'}(\sigma) &= (-1)^{-\chi d} \partial L_\sigma(\mu_I \otimes \mu_{c \times\Delta^n}) \\
 =&  L_\sigma \partial_{\otimes}(\mu_I \otimes \mu_{c \times\Delta^n})\\
=& L_\sigma(\partial \mu_I \otimes \mu_{c \times \Delta^n} - \mu_I \otimes \mu_I \otimes \partial \mu_{c \times \Delta^n})\\
=& L_\sigma( (\partial_0)_\#(\mu_{\Delta^0}) \otimes \mu_{c \times \Delta^n}) - L_\sigma( (\partial_1)_\#(\mu_{\Delta^0}) \otimes \mu_{c \times \Delta^n}) \\
&\; \; - L_\sigma(\mu_I \otimes EZ(\partial(\mu_c) \otimes \mu_{\Delta^n} + (-1)^m\mu_c \otimes \partial(\mu_{\Delta^n})))\\
=&  L_\sigma( (\partial_0)_\#(\mu_{\Delta^0}) \otimes \mu_{c \times \Delta^n}) - L_\sigma( (\partial_1)_\#(\mu_{\Delta^0}) \otimes \mu_{c \times \Delta^n}) \\
&\; \; -L_\sigma(\mu_I \otimes EZ(\partial(\mu_c) \otimes \mu_{\Delta^n})) - (-1)^mL_\sigma(\mu_I \otimes EZ(\mu_c \otimes \sum_i(-1)^i(\partial_i)_\#(\mu_{\Delta^{n-1}})))\\
=&  L_\sigma( (\partial_0)_\#(\mu_{\Delta^0}) \otimes \mu_{c \times \Delta^n}) - L_\sigma( (\partial_1)_\#(\mu_{\Delta^0}) \otimes \mu_{c \times \Delta^n}) \\
&\; \; -L_\sigma(\mu_I \otimes EZ(\partial(\mu_c) \otimes \mu_{\Delta^n})) - (-1)^m
\sum_i(-1)^i
L_\sigma(\mu_I \otimes EZ(\mu_c \otimes(\partial_i)_\#(\mu_{\Delta^{n-1}})))
\end{align*}

By Lemma \ref{Lsigma},
\begin{align*}
\partial \Lambda_{c, c'}(\sigma) 
=& \mathcal{ST}(c',\sigma ) - \mathcal{ST}(c, \sigma) -(-1)^m\sum_i (-1)^i L_{\sigma_i}(\mu_I \otimes \mu_{c \times \mu_{\Delta^{n-1}}}) 
\\
\;&  \hspace{2cm}
-L_\sigma(\mu_I \otimes EZ(\partial(\mu_c) \otimes \mu_{\Delta^n}))\\
=& \mathcal{ST}(c', \sigma) - \mathcal{ST}(c, \sigma ) 
-(-1)^{\chi d +m}\sum_i (-1)^i \Lambda_{c,c'}(\sigma_i)
-L_\sigma(\mu_I \otimes EZ(\partial(\mu_c) \otimes \mu_{\Delta^n}))\\
=& \mathcal{ST}(c', \sigma) - \mathcal{ST}(c, \sigma) 
-(-1)^{\chi d +m} \Lambda_{c,c'}(\partial \sigma)
-L_\sigma(\mu_I \otimes EZ(\partial(\mu_c) \otimes \mu_{\Delta^n}))
\end{align*}

So $\Lambda_{c,c}$ fails to be a chain homotopy between $\mathcal{ST}(c,-)$ and $\mathcal{ST}(c',-)$  exactly by the term 
$$-L_{\sigma,c,c'}(\mu_I \otimes EZ(\partial(\mu_c) \otimes \mu_{\Delta^n})).$$
 
 Recall that
 $$L_{\sigma, c, c'}(\mu_I \otimes EZ((\partial_{rs})_\#(\mu_{\partial_{rs} c}) \otimes \mu_{\Delta^{n}}))
 = L_{\sigma, \partial_{rs}c, \partial_{rs}c'}(\mu_I \otimes EZ(\mu_{\partial_{rs}c} \otimes \mu_{\Delta^{n}}))$$ from Lemma \ref{Lsigma}. 
 Assume $c_p$ and $c_q$ are two $m$-cells of $\sdbar$ such that $$\partial_{r_ps_p}c_p = \partial_{r_q s_q}c_q.$$ Let $c_p'$ and $c_q'$ be two other cells of $\sdbar$ which differ from $c_p$ and $c_q$ by a single compatible slide. Then
 $$L_{\sigma, c_p, c_p'}(\mu_I \otimes EZ((\partial_{r_ps_p})_\#(\mu_{\partial_{r_ps_p}c_p}) \otimes \mu_{\Delta^n})) = L_{\sigma, c_q, c_q'}(\mu_I \otimes EZ((\partial_{r_qs_q})_\#(\mu_{\partial_{r_qs_q}c_q}) \otimes \mu_{\Delta^n})).$$

Let  $z = \sum_j a_j c_j$ and $z' = \sum a_j c'_j$ be $m$-cycles such that $c_j$ and $c_j'$ differ by a single compatible slide for all $j$, then

\begin{align*}
\partial z &= \partial \left(\sum_j a_j c_j \right) = \sum_j a_j \partial(c_j) = 0 = \sum_{j,r,s}(-1)^{\epsilon(r,s)}a_j\partial_{rs}c_j = 0\\
&\implies 
 \sum_{j,r,s}(-1)^{\epsilon(r,s)}a_j L_{\sigma, c_j, c'_j}(\mu_I \otimes EZ((\partial_{rs})_\#(\mu_{\partial_{rs}c_j}) \otimes \mu_{\Delta^n}))
 = 0
\\
&\implies
 \sum_j a_j L_{\sigma, c_j, c'_j}(\mu_I \otimes EZ(\partial(\mu_{c_j}) \otimes \mu_{\Delta^n}))=0
 \end{align*}
 and

\begin{align*}
 \partial \Lambda_{z,z'}(\sigma) =&  \sum_j a_j c_j  \partial \Lambda_{c_j,c_j'}(\sigma) \\
=&\sum_j a_j c_j \Big(\mathcal{ST}(c_j', \sigma) - \mathcal{ST}(c_j, \sigma) 
-(-1)^{\chi d +m} \Lambda_{c_j,c_j'}(\partial \sigma)
\\
&\; \hspace{2cm}
-L_{\sigma, c_j, c'_j}(\mu_I \otimes EZ(\partial(\mu_{c_j}) \otimes \mu_{\Delta^n}))\Big)\\
=& \mathcal{ST}(z', \sigma) - \mathcal{ST}(z, \sigma ) 
-(-1)^{\chi d +m} \Lambda_{z,z'}(\partial \sigma).
 \end{align*}
 
 In particular, 
 $$\partial \Lambda_{z,z'} - (-1)^{- \chi d +m+1} \Lambda_{z,z'} \partial = \mathcal{ST}(z',-) - \mathcal{ST}(z,-)$$
 and $\Lambda_{z,z'}$ is a chain homotopy of  $\mathcal{ST}(z,-)$ and $\mathcal{ST}(z',-)$.
\end{proof}

\begin{cor}
If $z$ and $z'$ are slide-equivalent $m$-cycles in $\mathcal{C}_*(\sdbar)$, then $\mathcal{ST}(z,-)$ and $\mathcal{ST}(z', )$ induce the same map on homology: $$H_*(LM^k) \to H_{* + m - \chi d}(LM^\ell)$$.
\end{cor}

\begin{cor}
Let $R$ be a field and let $\mathcal{ST}_*: H_*(\sdbar) \otimes H_*(LM^k) \to H_{*-\chi d}(LM^\ell)$ be the map induced by $\mathcal{ST}$ on homology. Let $q: \sdbar \to \sdbarr$ be the quotient by slide-equivalence. Then $\mathcal{ST}_*$ factors through $$q_* \otimes id: H_*(\sdbar) \otimes H_*(LM^k) \to H_*(\sdbarr) \otimes H_*(LM^k).$$\end{cor}

In particular, we have a well-defined homology-level string topology construction $$H_*(\sdbarr) \otimes H_*(LM^k) \to H_{*-\chi d}(LM^\ell).$$

We now compare the operations $H_*(LM^k) \to H_{* -\chi d}(LM^\ell)$ arising from elements of $H_0(\mathscr{SD})$ to those of \ref{CG}.

Let $\mathscr{F}(g,k,\ell)$ be the space of marked metric fatgraphs of genus $g$ and $k+\ell$ boundary cycles, partitioned into $k$ inputs and $\ell$ outputs. The map $$\pi: Sull \to \mathscr{F}(g,k,\ell)$$ defined in \cite{CG} collapses ghost edges of a Sullivan chord diagram. Let $$Im(\pi) = RSull,$$ the space of reduced Sullivan chord diagrams. Godin shows \cite{GodinThesis} that $\pi$ is a homotopy equivalence. A nice outline of the proof is given in \cite{ChataurBordism}.

For a reduced Sullivan chord diagram with $k$ inputs, its edge-lengths may be scaled so that each input boundary cycle has length $1$ and that this rescaling is a homotopy equivalence. Let $R_1Sull$ be the space of reduced Sullivan chord diagrams whose inputs each have length $1$ and let $r_1: RSull \to R_1Sull$ be the deformation retraction given by rescaling input lengths. If $\Gamma$ is a string diagram such that $x_\Gamma \in SD$, then $\pi(x_\Gamma) \in R_1Sull$.

Denote the chord-contracting map  $\pi|_{SD}: SD \to R_1Sull$ by $\pi_{SD}$. If $\Gamma$ and $\Gamma'$ are slide-equivalent string diagrams, then $\pi_{SD}(x_\Gamma) = \pi_{SD}(x_{\Gamma'})$ so $\pi_{SD}$ factors through the quotient map $SD \to \mathscr{SD}$. Let $\pi_{\mathscr{SD}} \to R_1Sull$ be the induced map. 

\begin{prop}
The map $\pi_{\mathscr{SD}}: \mathscr{SD} \to R_1Sull$ is a homotopy equivalence.
\end{prop}

\begin{proof}
Recall that the space $\sdbar$ is a space with a regular cell complex structure and that slide-equivalence is cellular. We have a regular cell complex structure on $\sdbarr$ and $\mathscr{SD}$ is again union of open cells. The space $R_1Sull$ is also a union of open cells: each cell is labeled by a combinatorial type of reduced Sullivan chord diagram and parameters in a cell measure where vertices lie on input boundary cycles relative to the marked point and where marked points lie on output boundary cycles. These are exactly the parameters in a cell of $\mathscr{SD}$.

Recall also that a cell $c$ of $\mathscr{SD}$ is a product
$$\Delta^{n_1} \times \Delta^{n_2} \times \dots \times \Delta^{n_k} \times [0,1]^{N_\ell}$$
where $N_\ell$ is the number of output boundary components where the marked point lay in the interior of a directed chord edge.
Because $\pi_{\mathscr{SD}}$ collapses chords, the image $\pi_{\mathscr{SD}}(c)$  of $c$ is the product
$$\Delta^{n_1} \times \Delta^{n_2} \times \dots \times \Delta^{n_k}.$$

Let $\mathscr{U}$ be a cover of $R_1Sull$ by open sets $U$ such that every $U$ is a contractible neighborhood of an open cell. Then $\pi_{SD}^{-1}(U)$ is an open neighborhood of a cell in $\mathscr{SD}$ and $\pi_{SD}: \pi_{SD}^{-1}(U) \to U$ is a homotopy equivalence, in particular, a weak homotopy equivalence. By Corollary 1.4 of \cite{May}, $\pi_{\mathscr{SD}}: \mathscr{SD} \to R_1Sull$ is a weak equivalence. Whitehead's theorem implies that it is a homotopy equivalence.

\end{proof}

We have proved that if $\pi_\mathscr{SD}': R_1Sull \to  \mathscr{SD}$ is a homotopy inverse for $\pi_\mathscr{SD}$ then
$$\pi_\mathscr{SD}'  \circ r_1 \circ \pi: Sull \to \mathscr{SD}$$
is a homotopy equivalence.

\begin{cor}\label{SDSull}
$Sull$ and $\mathscr{SD}$ are homotopy equivalent.
\end{cor}

\begin{theorem}\label{SullSDcommute}
Let 
\[i: Sull \stackrel{\cong}{\to} \mathscr{SD} \hookrightarrow \sdbarr\]
be the composition of inclusion and homotopy equivalence. The following diagram commutes.
\begin{displaymath}
    \xymatrix{      H_0(Sull) \otimes H_*(LM)^{\otimes k} \ar[drr]^{\mu} \ar[dd]_{i_0 \otimes id}
 \\ 
     &&H_{*-\chi d}(LM)^{\otimes \ell}\\
                                        H_0(\sdbarr) \otimes H_*(LM)^{\otimes k} \ar[urr]_{\widetilde{\mathcal{ST}}_*}  \\
                 }
\end{displaymath}

\end{theorem}

\begin{proof}
The spaces $\mathscr{SD}$ and $\sdbarr$ and $Sull$ are all connected. Let $c_{[\Gamma]}$ be a 0-cell  in $\mathscr{SD}$ representing a generator of $H_0(Sull)$. Then $i(c_{[\Gamma]})$ represents the generator of $H_0(\sdbarr)$.
Together, Theorem \ref{CG} and Proposition \ref{slideprop} show that $$\mu(c_{[\Gamma]}, ) = \widetilde{\mathcal{ST}}_*(i(c_{[\Gamma]}), ).$$ That is, the diagram commutes when evaluated on $c_{[\Gamma]}$. 
Since all the maps in the diagram are linear, the diagram commutes.

\end{proof}

\section{The TQFT structure on homology}

Recall that in Theorem \ref{SullSDcommute} we saw that we recover string topology operations on $H_*(LM)$ arising from homology classes in $H_0(Sull)$. Gluing of Sullivan chord diagrams is defined in \cite{CG} and used to define a positive boundary topological quantum field theory.
In this section we define gluing of string diagrams and show that induced operations on homology respect this gluing for $H_0(\sdbarr)$.

Gluing of slide-equivalence classes of string diagrams is defined as follows.

Let $x_{\Gamma_1}  \in \sdbar(g_1, k_1, \ell_1)$ and $x_{\Gamma_2} \in \sdbar(g_2, k_2, \ell_2)$. Let $\mathfrak{o}_1= \{ o_1, o_2, \dots, o_{\ell_1}\}$ be the set of outputs of $\Gamma_1$, $\mathfrak{i}_2=\{i_1, i_2, \dots, i_{k_2}\}$ be the set of inputs of $\Gamma_2$ and $\mathfrak{s} \subset \mathfrak{o}_1 \times \mathfrak{i}_2 $  be a subset where any element of $\mathfrak{o}_1$ or $\mathfrak{i}_2$ appears at most once as a coordinate of an ordered pair. For $(o_r, i_s) \in \mathfrak{s}$, we identify output the output $o_r$ with the input $i_s$ according to their parametrizations by $S^1$ for all $(o_r, i_s) \in \mathfrak{s}$. Notice that this will usually involve a rescaling of $i_s$ to have the same length as $o_r$. The result of the identifications need not be a string diagram: chord endpoints of $\Gamma_2$ may be identified with points in the interiors of chords of $\Gamma_1$.  

Rather than gluing string diagrams, we glue slide-equivalence classes instead. 
\begin{definition}
Let $\Gamma_i$ represent the slide-equivalence class $[\Gamma_i]$ (corresponding to $x_{\Gamma_i}  \in \sdbar(g_, k_i, \ell_i)$ and $x_{[\Gamma_i]}  \in \sdbarr(g_i, k_i, \ell_i)$) $i \in \{1,2\}$). Identify $o_r$ of $\Gamma_1$ with $i_s$ of $\Gamma_2$ for all $(o_r, i_s) \in \mathfrak{s}$ as above. If any chord endpoint $v$ of $\Gamma_2$ is identified with a point in interior of a chord $e$ of $\Gamma_1$, then slide $v$ to one endpoint or the other of $e$ so that it coincides with a vertex on an input circle of $\Gamma_1$.  The result is a string diagram $\Gamma_1 \# \Gamma_2$ of type $(g_1+g_2+|\mathfrak{s}|-1, k_1+k_2-|\mathfrak{s}|,\ell_1+\ell_2-|\mathfrak{s}|)$. We order inputs by first listing inputs of $\Gamma_1$ followed by inputs of $\Gamma_2$ that do not appear as a coordinate in $\mathfrak{s}$ and order the outputs  by first listing the outputs of $\Gamma_1$ that do not appear as a coordinate in $\mathfrak{s}$ followed by the outputs of $\Gamma_2$.  The slide-equivalence class $[\Gamma_1\#_\mathfrak{s} \Gamma_2]$ is independent of the representatives $\Gamma_i$. Therefore, the following map is well defined.
\begin{align*}
 \#_\mathfrak{s}:\sdbarr(g_2, k_2, \ell_2) \times \sdbarr(g_1, k_1, \ell_1) &\to \sdbarr(g_1+g_2+|\mathfrak{s}|-1, k_1+k_2-|\mathfrak{s}|,\ell_1+\ell_2-|\mathfrak{s}|)\\
 (x_{[\Gamma_2]}, x_{[\Gamma_1]}) &\mapsto x_{[\Gamma_1\#_\mathfrak{s} \Gamma_2]}.
\end{align*}
\end{definition}

We would like to say that such operations give
$$\bigsqcup_{\gkl} \sdbarr \gkl$$
the structure of a properad \cite{Vallette} but composition of such operations need not be associative. However, for any $\mathfrak{s}$, $\#_\mathfrak{s}$ is a cellular map and composition of induced maps on cellular chains, and hence on homology, is associative.

\begin{definition}
Let $ (\#_\mathfrak{s})_\#$ and $ (\#_\mathfrak{s})_*$ denote
 the maps induced by $\#_\mathfrak{s}$ on cellular chains and cellular homology.
 To be explicit:
\begin{align*}
( \#_\mathfrak{s})_\#:\mathcal{C}_*(\sdbarr(g_2, k_2, \ell_2)) &\otimes \mathcal{C}_*(\sdbarr(g_1, k_1, \ell_1)) \stackrel{EZ}{\to}
\mathcal{C}_*(\sdbarr(g_2, k_2, \ell_2)) \times \sdbarr(g_1, k_1, \ell_1)) \\
&\to
 \mathcal{C}_*(\sdbarr(g_1+g_2+|\mathfrak{s}|-1, k_1+k_2-|\mathfrak{s}|,\ell_1+\ell_2-|\mathfrak{s}|))\\
 ( \#_\mathfrak{s})_*:H_*(\sdbarr(g_2, k_2, \ell_2)) &\otimes H_*(\sdbarr(g_1, k_1, \ell_1)) \stackrel{EZ}{\to}
H_*(\sdbarr(g_2, k_2, \ell_2)) \times \sdbarr(g_1, k_1, \ell_1)) \\
&\to
H_*(\sdbarr(g_1+g_2+|\mathfrak{s}|-1, k_1+k_2-|\mathfrak{s}|,\ell_1+\ell_2-|\mathfrak{s}|))
 \end{align*}
\end{definition}

We might hope that 
the maps $(\#_\mathfrak{s})_\#$ (respectively $(\#_\mathfrak{s})_*$) give
$$\mathcal{C}_*(\sdbarr) = \bigsqcup_{\gkl} \mathcal{C}_*(\sdbarr \gkl) \left( \textrm{ respectively } H_*(\sdbarr)= \bigsqcup_{\gkl} H_*(\sdbarr \gkl) \right)$$
respectively the structure of a properad
and that $\widetilde{\mathcal{ST}}$ would give $C_*(LM)$ the structure of an algebra over the properad $\mathcal{C}_*(\sdbarr)$ (respecively that
$\widetilde{\mathcal{ST}}_*$ would give $H_*(LM)$ the structure of an algebra over the properad $H_*(\sdbarr)$).
This not need be the case. There exist cellular chains of $\sdbar$ and 
whose composition is $0$ in $\mathcal{C}_*(\sdbarr)$ for dimension reasons, but the composition of the corresponding string topology operations is not identically $0$. This situation occurs, for example, when the first chain is a cell of string diagrams that have an output boundary cycle made up only of directed edges corresponding to chords and that output is identified with an input of the second chain in $\#_\mathfrak{s}$.
 However, we do have the following result.

\begin{prop}
Let $k_2 = \ell_1 = |\mathfrak{s}|$,
$\sdbarr_1 = \sdbarr(g_1, k_1, \ell_1)$, 
$\sdbarr_2 = \sdbarr(g_2, k_2, \ell_2)$, and
$\sdbarr_3 = \sdbarr(g_1+g_2, k_1, \ell_2)$.
Then the following diagram commutes.
\begin{displaymath}
    \xymatrix{      H_0(\sdbarr_2) \otimes  H_0(\sdbarr_1)  \otimes H_*(LM)^{\otimes k_1}  
    \ar[rr]^{id \otimes \mathcal{ST}_*}
     \ar[dd]_{(\#_\mathfrak{s})_0 \otimes id} 
     && H_0(\sdbarr_2) \otimes H_{*}(LM)^{\otimes k_2}
     \ar[dd]^{\mathcal{ST}_* \otimes id}
   \\
   & 
   \\
  H_0(\sdbarr_3) \otimes H_*(LM)^{\otimes k_1} 
  \ar[rr]_{\mathcal{ST}_*}
  &&  H_*(LM)^{\otimes \ell_2} \\
                 }
\end{displaymath}
\end{prop}

Implicit in the diagram are the appropriate degree shifts.

\begin{proof}

Gluing of Sullivan chord diagrams is defined slightly differently then gluing of slide-equivalence classes: outputs of $\Gamma_1$ are identified with inputs of $\Gamma_2$ according to their parametrizations. Pointwise, gluing of Sullivan chord diagrams need not be continuous or well-defined, 
but there is a well-defined induced map on 0-dimensional homology:
$$(\#)_0: H_0(Sull_2) \otimes H_0(Sull_1) \to H_0(Sull_2 \times Sull_1) \to H_0(Sull_3)$$
and operations arising from 0-dimensional homology classes are well defined.

Theorem 6 of \cite{CG} implies that the following diagram commutes.

\textbf{Diagram (1):}
\begin{displaymath}
    \xymatrix{      H_0(Sull_2) \otimes  H_0(Sull_1)  \otimes H_*(LM)^{\otimes k_1}  
    \ar[rr]^{id \otimes \mu_1}
     \ar[dd]_{(\#)_0 \otimes id} 
     && H_0(Sull_2) \otimes H_{*}(LM)^{\otimes k_2}
     \ar[dd]^{\mu_2 \otimes id}
   \\
   &
   \\
  H_0(Sull_3) \otimes H_*(LM)^{\otimes k_1} 
  \ar[rr]_{\mu_3}
  &&  H_*(LM)^{\otimes \ell_2} \\
                 }
\end{displaymath}

By Corollary \ref{SDSull}, the horizontal maps in the following diagram are isomorphisms. Each of the vertical maps is given by $1\otimes1 \mapsto 1$.

\textbf{Diagram (2):}
\begin{displaymath}
    \xymatrix{      H_0(Sull_2) \otimes  H_0(Sull_1)
    \ar[rr]^{(i_2)_0 \otimes (i_1)_0}
     \ar[dd]_{(\#)_0} 
     && H_0(\sdbarr_2) \otimes H_0(\sdbarr_1)
          \ar[dd]^{(\#_\mathfrak{s})_0}
   \\
   & 
   \\
  H_0(Sull_3)  
  \ar[rr]_{i_3}
  &&  H_0(\sdbarr_3)\\
                 }
\end{displaymath}

We have a large commutative diagram. The outer square is diagram (1) above. The inner square is the desired commutative diagram. The square on the left commutes by diagram (2) above and the other three squares commute by Theorem \ref{SullSDcommute}. This implies the desired diagram commutes.

\begin{displaymath}
    \xymatrix{      \bullet
    \ar[rrr]^{id \otimes \mu_1}
    \ar[dr]^{i_0 \otimes i_0 \otimes id}
    \ar[ddd]_{\#_0 \otimes id}
    &&&
    \bullet
    \ar[dl]_{i_0 \otimes id}
    \ar[ddd]^{\mu_2 \otimes id}
\\
   & \bullet
    \ar[r]
    ^{id \otimes \widetilde{\mathcal{ST}}_*}
     \ar[d]_{\#_{\mathfrak{s}_0} \otimes id} 
     & \bullet
          \ar[d]^{\mathcal{\widetilde{ST}}_* \otimes id}
     &
   \\
&  \bullet
  \ar[r]_{\widetilde{\mathcal{ST}}_*}
  & \bullet \ar[dr]_{id}
  \\
   \bullet \ar[ur]_{i_0 \otimes id}
   \ar[rrr]_{\mu_3}
   &&&
   \bullet \ar[ul]
                 }
\end{displaymath}
\end{proof}

We have shown that the operations induced by elements of $H_0(\sdbarr)$ on the homology of the loop space agree with those in \cite{CG} and that these operations respect gluing. Thus we
 obtain a Frobenius algebra structure of $H_*(LM; R)$, in the sense of \cite{CG}, when $R$ is a field.
 
\begin{cor}
The isomorphisms $H_0(Sull \gkl) \to H_0(\sdbarr \gkl)$ induced by the inclusions $i: Sull \gkl \to \sdbarr \gkl$ induce an isomorphism
$$H_*(LM) \to H_*(LM)$$ of Frobenius algebras without counit.
\end{cor}

\bibliographystyle{amsalpha}
\bibliography{katenatebiblio}
\end{document}